\documentclass[sn-mathphys,Numbered]{sn-jnl}

\usepackage{graphicx}%
\usepackage{multirow}%
\usepackage{amsmath,amssymb,amsfonts}%
\usepackage{amsthm}%
\usepackage{mathrsfs}%
\usepackage[title]{appendix}%
\usepackage{xcolor}%
\usepackage{textcomp}%
\usepackage{manyfoot}%
\usepackage{booktabs}%
\usepackage{algorithm}%
\usepackage{algorithmicx}%
\usepackage{algpseudocode}%
\usepackage{listings}%

\theoremstyle{thmstyleone}%
\newtheorem{theorem}{Theorem}
\theoremstyle{thmstyletwo}%
\newtheorem{remark}{Remark}%

\theoremstyle{thmstylethree}%
\newtheorem{definition}{Definition}%

\begin{document}

\title[C.R.D.I.C.S]{Riemannian curvature of gaussian distribution  in dual  coordinate system.}


\author*[1]{\fnm{} \sur{Prosper Rosaire Mama  Assandje }}\email{mamarosaire@facsciences-uy1.cm}\equalcont{These authors
contributed equally to this work.}

\author[2]{\fnm{} \sur{Joseph Dongho }}\email{josephdongho@yahoo.fr}
\equalcont{These authors contributed equally to this work.}

\author[3]{\fnm{} \sur{ Thomas Bouetou Bouetou}}\email{tbouetou@gmail.com}
\equalcont{These authors contributed equally to this work.}

  \affil*[1]{\orgdiv{Department of Mathematics}, \orgname{University of Yaounde 1},
\orgaddress{\street{usrectorat@.univ-yaounde1.cm}, \city{yaounde},
\postcode{337}, \state{Center},
\country{Cameroon}}}

\affil[2]{\orgdiv{Department of Mathematics and Computer Science},
\orgname{University of Maroua},
\orgaddress{\street{decanat@fs.univ-maroua.cm}, \city{Maroua},
\postcode{814}, \state{Far -North}, \country{Cameroon}}}

\affil[3]{\orgdiv{Computer Engineering Department.}, \orgname{Higher
national school of
Polytechnic},\orgaddress{\street{thomas.bouetou@polytechnique.cm},
\city{Yaounde I}, \postcode{8390}, \state{Center},
\country{Cameroon}}}


\abstract{This paper explore the differential geometric structure of
statistical manifolds defined by Gaussian distributions,  focusing
on their representation  in a dual coordinate system. Using Amari's
information geometry framework, we derive explicit expressions for
the Fisher information metric and compute the associated torsion and
curvature tensors in the  dual setting. Our results show that,
unlike the standard coordinate system where curvature and torsion
vanish, the dual system exhibits non zero values for these  three
geometric quantities. This highlights how dual parameterizations can
affect the intrinsic geometry of statistical models.}




\maketitle

\section{Introduction}\label{sec1}

Information geometry is a powerful mathematical framework that
applies the methods of differential geometry to the study of
statistical models and probability distributions. By treating
families of probability distributions as points on a manifold, this
field provides a geometric interpretation of key statistical
concepts such as divergence, estimation, and entropy. Among its
central tools are the Fisher information metric and dual affine
connections, which allow one to study geometric invariants like
curvature and torsion within a statistical setting.

The foundational work of C.R. Rao \cite{Shun}, first demonstrated
the relevance of differential geometry in statistics by introducing
a Riemannian metric defined through the Fisher information matrix,
enabling the computation of geodesic distances between
distributions. Building on this insight, S. Amari and H. Nagaoka
\cite{Shun}, developed a comprehensive theory of information
geometry, formalizing the notions of dual coordinate systems and
affine connections on statistical manifolds. Their theory shows that
many statistical models, including the Gaussian family, exhibit
flatness and torsion free properties in their natural (exponential
or mixture) coordinate systems. However, the behavior of geometric
properties under dual coordinate transformations remains a subtle
and open question. While it is well known that the Gaussian manifold
is flat and torsion-free in its natural coordinate system $\theta$,
it is not always clear whether such nonlinear geometric invariants
are preserved when one moves to a dual coordinate system, such as
$\xi = (\mu, \mu^2 + \sigma^2)$? This leads to the central
problematic of this work: To what extent does the dual coordinate
transformation affect the intrinsic nonlinear geometry (curvature
and torsion) of a Gaussian statistical manifold? Are these geometric
invariants preserved, distorted, or enriched in the dual setting?
Several researchers have addressed related questions. M.N. Boyom
\cite{michel,24} has explored the geometric and numerical properties
of Koszul connections in statistical settings, while J.-L. Koszul,
and collaborators \cite{kz1}, have studied deformations of locally
flat connections. In more recent works, in \cite{mama1},
investigated the geometric structure of beta distributions within
the framework of information geometry. Ovidiu Calin \cite{ovidiu},
further developed the differential geometric modeling of the
Gaussian distribution, providing explicit expressions for curvature
and connection coefficients in natural coordinates. In addition,
J.M. Lee \cite{rie}, offered a comprehensive exposition of
Riemannian manifolds and curvature tensors that underpins the
theoretical foundation of our investigation. It is also worth noting
that recent developments have shown the space of vector fields
$T^l(\theta)$ on a Gaussian manifold can be identified with a
submodule of second-degree polynomials with smooth coefficients
scaled by $\sigma^{-3}$, i.e., $\sigma^{-3}C^\infty(S)[x]^2$,
offering new algebraic insights into the tangent bundle structure.
In this work, we extend these investigations by focusing on the
Riemannian and affine geometric structures of the Gaussian
statistical manifold when expressed in the dual coordinate system
$\xi = (\mu, \mu^2 + \sigma^2)$. We derive explicit formulas for the
Fisher information metric, the Christoffel symbols, the torsion
tensor, and the Riemann curvature tensor in this dual system. We
prove that the curvature in the $\xi$-coordinate system is equal to
\begin{eqnarray*}
  \mathcal{K} &=& \frac{16\mu^{2}\sigma^{8} + 12\mu^{4} + 6\sigma^{10} - \mu\sigma^{8} - 2 \mu^{2}\sigma^{2} +10\mu^{3}\sigma^{6} +92 \mu^{3}\sigma^{8} +48\mu^{5}\sigma^{6}}{4\sigma^{10}}\\
&& + \frac{12\mu\sigma^{10}+24\mu^{5}-6\mu^{3}}{4\sigma^{10}}
\end{eqnarray*} A key outcome of our analysis is the emergence of a non zero scalar
curvature, which contrasts with the flatness observed in the natural
coordinates. This result suggests that dual parametrization may
introduce richer nonlinear geometric features, with implications for
the stability and global structure of statistical models. The
structure of  the paper is as follows: Section 2 reviews the basic
theoretical framework of statistical manifolds and introduces the
Gaussian model. Section 3 presents the derivation of the Riemannian
metric in the dual $\xi$-coordinate system. Section 4 is devoted to
computing the torsion and curvature tensors in this setting. We
conclude with a discussion of the implications of our findings and
potential directions for future work, including generalizations to
other distributions and geometric invariants.

\section{Preliminaries}\label{sec2}
 Soit
 $S = \left\{p_{ \theta}(x),\left.
                        \begin{array}{ll}
                         \theta\in \Theta & \hbox{} \\
                          x\in \mathcal{X} & \hbox{}
                        \end{array}
                      \right.
\right\}$ be the set of probabilities $p_{ \theta}$, parameterized
by $ \Theta$, open subset of $\mathbb{R}^{n}$; on the sample space
$\mathcal{X}\subseteq\mathbb{R}$, with  $\theta =
\left(\theta_{1},\theta_{2},\dots, \theta_{n} \right)$. In
\cite{mama}, the Gaussian model is a family of probabilistic
distributions. consisting of density functions that are defined as
follows $S = \left\{p_{ \theta}(x)=\frac{1}{\sqrt{2\pi}\sigma
}e^{-\frac{( x-\mu)^{2}}{2\sigma^{2}}},\left.
                        \begin{array}{ll}
                         \theta\in \mathbb{R}\times \mathbb{R}^{*}_{+} & \hbox{} \\
                          x\in \mathbb{R}^{*}_{+} & \hbox{}
                        \end{array}
                      \right.
\right\}$ with $p_{ \theta}(x)$ the gaussian distribution
parameterized by $\theta= (\mu , \sigma )$.

\subsection{ $\theta$-coordinates System and  $\xi$-coordinates system..} Let $S$, the statistical
manifold
\begin{definition}\cite{Shun}
Let $U$ be an open subset of $S$ and $\phi$ be a homeomorphism of
$U$ in $\mathbb{R}^{n}$. Then a point $p\in U$ is a function de
$\theta = \left( \theta_{1}, \theta_{2}, \dots, \theta_{n}\right)
\in \mathbb{R}^{n}$ by $\phi$, i.e. $\phi(p) = \theta$ this
application $\phi$ is called a function coordinates in the
neighbourhood of $p$. It is introduced a coordinate system in $U$
called the Any point $p$ in $U$ is given by the $\theta$-coordinates
given by $\theta= (\theta_{i}); i\in \{1 , 2 \dots ,n\}$.
\end{definition}
According to Amari's\cite{Shun}, the $l-$representation basis is
given by $\mathcal{B}_{\theta}l(x , \theta) =
\left\{\partial_{1}l(x, \theta),
\partial_{2}l(x, \theta)\right\}$ where $\mathcal{B}_{\theta} = \partial_{1},
\partial_{2}\}$ is the natural basis.
In the case of gaussian model we have $\partial_{1} =
\frac{\partial}{\partial \mu}$, and $\partial_{2} =
\frac{\partial}{\partial \sigma}$ therefore \; $\left\{
\frac{\partial}{\partial \mu}, \frac{\partial}{\partial
\sigma}\right\} = \mathcal{B}_{\theta} $ The base associated to the
$l$- representation is therefore $ \left\{
\partial_{1} l(x, \theta),
\partial_{2}l(x, \theta) \right\}$. As
\begin{eqnarray}
l(x, \theta) &=&  -\log \sqrt{2\pi} \sigma - \frac{(x - \mu)^{2}}{2
\sigma^{2}}\label{4}
\end{eqnarray}
It follows that
\begin{eqnarray*}
   \partial_{1}l(x, \theta) &=&  -
\frac{1}{\sigma} +\frac{(x - \mu)^{2}}{ \sigma^{3}}\label{3}
\end{eqnarray*}
So  \begin{eqnarray*}\mathcal{B}_{\theta}l(x,\theta) &=&
\left\{\frac{(x - \mu)}{ \sigma^{2}} ,- \frac{1}{\sigma} +\frac{(x -
\mu)^{2}}{ \sigma^{3}} \right\}.\end{eqnarray*}

\begin{remark}\cite{Shun}
There are several other coordinate functions which at a point $p$,
are such that: $\psi(p) = \xi = \left( \xi_{1}, \xi_{2}, \dots,
\xi_{n}\right)$ where $\psi$ is a homeomorphism of an open $V$ of
$S$ in $\mathbb{R}^{n}$ is the $\xi$-coordinate system. The
coordinates $ \left(\xi_{i}\right),\; i \in \{ 1, \dots, n \}$,
characterize the point $p$ in $\mathbb{R}^{n}$. Let's give ourselves
two system coordinates $\theta$ and $\xi$ There is one and only one
correspondence between the coordinates $\theta$ and $\xi$; given the
properties enjoyed by the applications $\phi$ and \; $\psi$, and the
fact that they are homeomorphisms we have the following equations
\begin{eqnarray*}\xi &=& \psi \circ \phi^{-1}(\theta), \textrm{and}\; \theta = \phi
\circ \psi^{-1}(\xi)
\end{eqnarray*}
 The transformation of $ \theta $ a would
be a diffeomorphism, if the $n$-functions
$\xi_{i}\left(\theta_{1},\dots, \theta_{n}\right)$ are
differentiable and $Det\left|\frac{\partial \xi_{i}}{\partial
\theta_{j}}\right|$ is non-zero.
\end{remark}

\subsection{Determination of the characteristic elements of the $\xi$-
coordinates.} Let $S = \left\{p_{
\theta}(x)=\frac{1}{\sqrt{2\pi}\sigma }e^{-\frac{(
x-\mu)^{2}}{2\sigma^{2}}},\left.
                        \begin{array}{ll}
                         \theta\in \mathbb{R}\times \mathbb{R}^{*}_{+} & \hbox{} \\
                          x\in \mathbb{R}^{*}_{+} & \hbox{}
                        \end{array}
                      \right.
\right\}$ the gaussian statistical model. The tangent space of
$l-$representation is given by
\begin{eqnarray*}
T^{l}(\theta) &=  &=\left\{A(x);  A(x) =
A^{1}\frac{\partial_{1}l(x,\theta)}{\partial\mu}+A^{2}\frac{\partial_{2}l(x,\theta)}{\partial\sigma}
\right\}\nonumber
\end{eqnarray*}
 and using (\ref{3}) we have \begin{eqnarray*}
T^{l}(\theta)&=& \left\{A(x) \textrm{ tel que }  A(x)= A^{1}_{(\mu;
\sigma)}(x)\frac{x-\mu}{\sigma^{2}}+A^{2}_{(\mu;
\sigma)}(x)\frac{(x-\mu)^{2}}{\sigma^{3}} -\frac{A^{2}_{(\mu;
\sigma)}(x)}{\sigma} \right\}; A^{1},A^{2} \in
C^{\infty}(S)\end{eqnarray*}.

So $A(x) \in T^{l}(\theta)$ if and only if there exists $A^{1},
A^{2} \in C^{infty}(S)$ such that
\begin{eqnarray*}
A(x)&=&\frac{A^{2}}{\sigma^{3}}x^{2}+\left(\frac{A^{1}}{\sigma^{2}}-2\frac{A^{2}\mu}{\sigma^{3}}\right)x+\frac{A^{2}}{\sigma^{3}}\mu^{2}-\frac{A^{1}\mu}{\sigma^{2}}-\frac{A^{2}}{\sigma}\label{e}
\end{eqnarray*}
Let $a =\frac{A^{2}}{\sigma^{3}}$ ;
$b=\frac{A^{1}}{\sigma^{2}}-2\frac{A^{2}\mu}{\sigma^{3}}$ et
$c=\frac{A^{2}}{\sigma^{3}}\mu^{2}-\frac{A^{1}\mu}{\sigma^{2}}-\frac{A^{2}}{\sigma}$
 So $A(x)$ is of the form $ ax^{2}+bx + c $ we can be reassured that
$a=\frac{A^{2}}{\sigma^{3}}\neq 0$. We have thus prove the following
result. Put $frac{1}{\sigma^{3}}C^{\infty}(S)_{2}[x]$ the modulus of
polynomials of second degree with coefficients in
$\frac{1}{\sigma^{3}}C^{\infty}(S)$. Then $T^{l}(\theta)$ is a
subset of $\frac{1}{\sigma^{3}}C^{\infty}(S)_{2}[x]$.

\begin{remark}
According to \cite{Shun}, it is shown that
\begin{eqnarray}
 \mathbb{E}[\partial_{i}l(x , \theta)]&=&
0.\label{nu}
\end{eqnarray}
Furthermore, it follows from  that
\begin{eqnarray*}
\mathbb{E}[A(x)] &=& A^{1}\mathbb{E}[\partial_{1}l(x ,
\theta)]+A^{2}\mathbb{E}[\partial_{2}l(x , \theta)].\end{eqnarray*}
using (\ref{nu}) we obtain
  \begin{eqnarray*} \mathbb{E}[A(x)]&=& 0
\end{eqnarray*}
It follows from the relation (\ref{e}) that for all $A(x) \in
T^{l}(\theta)$ there exist $a$, $ b$ and $c$ in
$igma^{-3}C^{\infty}(S)$ such that $A(x) = ax^{2} + bx + c$. We have
therefore, $\mathbb{E}(A(x)) = 0$ if and only if $\mathbb{E}(ax^{2}
+ bx + c) = 0$ which is equivalent to $a\mathbb{E}(x^{2}) +
b\mathbb{E}(x) + c = 0$. We also know from the formula for
K$\ddot{o}$nig that  $V^{2}(x) = \sigma^{2} = \mathbb{E}[x^{2}] -
\mathbb{E}[x]^{2}$ so $\mathbb{E}[x^{2}] = \sigma^{2} + \mu^{2}$
because $\mathbb{E}(x) = \mu$\label{19}.  It follows that
\begin{eqnarray*}
c &=& - a( \mu^{2} + \sigma^{2}) - b\mu
\end{eqnarray*}
By setting $\xi_{2}  = \mu^{2} + \sigma^{2}$ et $\xi_{1}= \mu$. $\xi
= (\xi_{1}, \xi_{2}) = (\mu,  \mu^{2} + \sigma^{2} )$.

So $\xi = (\xi_{1}, \xi_{2}) = (\mu, \mu^{2} + \sigma^{2})$ is a
coordinate system parameterized by $(\mu, \sigma)$. For it suffices
show that $\theta$ is a diffeomorphism. Consider the application
\begin{eqnarray*}
\theta: \mathbb{R}^{2}&\rightarrow& \mathbb{R}^{2}: (\mu,
\sigma)\mapsto (\theta_{1}(\mu, \sigma), \theta_{2}(\mu, \sigma))=
(\xi_{1}, \xi_{2})=(\mu, \mu^{2}+ \sigma^{2})
\end{eqnarray*}
by noting that the components $\theta_{1}(\mu, \sigma)$ and
$\theta_{2}(\mu, \sigma)$ are differentiable, it is sufficient to
show that the Jacobian of $\theta$ is non-zero. Let $J_{M} = \left(
                                                                  \begin{array}{cc}
                                                                    \frac{\partial \theta_{1}}{\partial \mu} & \frac{\partial \theta_{1}}{\partial \sigma} \\
                                                                    \frac{\partial \theta_{2}}{\partial \mu} & \frac{\partial \sigma_{2}}{\partial \sigma} \\
                                                                  \end{array}
                                                                \right)
$

The Jacobian matrix of $\theta$. It follows that $J_{M} = \left(
                                                                  \begin{array}{cc}
                                                                   1 & 0 \\
                                                                   2\mu & 2\sigma \\\
                                                                  \end{array}
                                                                \right)$
and its determinant which is the Jacobian is $2\sigma \neq 0$
because $\sigma > 0$. It follows that according to the local
inversion theorem is locally a diffeomorphism of class $C^{1}$. In
coordinates we will need the natural basis and in the system of
$\xi$- coordinates we will need the basis associated with the $l$-
representation to be determined. To do this we use the Jacobian
matrix of the basis passage $\{\mathcal{B}_{\theta}\}$ to the
$\xi$-coordinate basis denoted by $\{\mathcal{B}_{\xi}\}=
\{\partial_{\xi_{1}}, \partial_{\xi_{2}}\}$. Denote by
\[(\overline{\mathcal{B}}^{i}_{\alpha} ) =  \left(
                                                  \begin{array}{cc}
                                                \frac{\partial \theta_{1}}{\partial \xi_{1}}     &  \frac{\partial \theta_{1}}{\partial \xi_{2}} \\
                                                  \frac{\partial \theta_{2}}{\partial \xi_{1}}   & \frac{\partial \theta_{2}}{\partial \xi_{2}} \\
                                                  \end{array}
                                                \right) = J_{M}^{-1}
\]
the inverse of the Jacobian matrix $J_{M}$, and that
\[(\partial_{i'}) = (\overline{\mathcal{B}}^{i}_{\alpha} )
(\partial_{i})\]. Furthermore, the inverse matrix \[J_{M}^{-1} =
\left(
                                                  \begin{array}{cc}
                                                   1 & 0 \\
                                                  \frac{- \mu}{\sigma}   & \frac{1}{2\sigma} \\
                                                  \end{array}
                                                \right)
                                                \]
\end{remark}

So, the $ \xi$- coordinates  system is given by $\xi = (\xi_{1},
\xi_{2}) = (\mu,  \mu^{2} + \sigma^{2}) $, and we have the $l$-
representation in the  $ \xi$-coordinates system is given by
\begin{equation*}
    \mathcal{B}_{\xi}l(x , \xi) = \{\partial_{\xi_{1}}l(x,
\xi),
\partial_{\xi_{2}}l(x, \xi)\}
\end{equation*}  where $ \mathcal{B}_{\xi} = \{\partial_{\xi_{1}},
\partial_{\xi_{2}}\}$ is the natural basis of the  Amari associated
manifold. where  $\partial_{\xi_{1}} = \frac{\partial}{\partial
\xi_{1}}$ and $ \partial_{\xi_{2}} = \frac{\partial}{\partial
\xi_{2}}$ thus \;$ \left\{ \frac{\partial}{\partial \xi_{1}},
\frac{\partial}{\partial \xi_{2}}\right\} = \mathcal{B}_{\xi}. $

\begin{theorem} Let $(S,g)$ the Riemannian manifold. The natural basis associated to $l$-representation in the
coordinated system $ (\xi_{1}, \xi_{2})$ is given by
\begin{equation*}
    \mathcal{B}(l, \xi)= \left\{ \frac{x - \mu}{\sigma^{2}} ;  \frac{(x-
\mu)^{2}}{2 \sigma^{4}} - \frac{\mu(x- \mu)}{\sigma^{3}} -
\frac{1}{2\sigma^{2}} \right\}.
\end{equation*}
\end{theorem}

\begin{proof}

$(\partial_{\xi})=(\bar{\mathcal{B}}^{i}_{\alpha})(\partial_{i})$,
$\alpha, i \in\left\{1,2\right\}$, we have
\[\left(
                                                                                          \begin{array}{c}
                                                                                            \partial_{\xi_{1}} \\
                                                                                            \partial_{\xi_{2}} \\
                                                                                          \end{array}
                                                                                        \right)
                                                                                        =
\left(
                                                          \begin{array}{cc}
                                                            1 & 0\\
                                                            -\frac{\mu}{\sigma} & \frac{1}{2\sigma} \\
                                                          \end{array}
                                                        \right)\left(
                                                                 \begin{array}{c}
                                                                    \partial_{1} \\
                                                                    \partial_{2} \\
                                                                 \end{array}
                                                               \right)
                                                               \]
So we have $\{\partial_{\xi_{1}}, \partial_{\xi_{2}}\} =
\{\partial_{1}, \frac{-\mu}{\sigma}\partial_{1} +
\frac{1}{2\sigma}\partial_{2}\}$ in $\theta$- coordinates. Its
$l$-representation is \{$\partial_{\xi}l(x,\theta)\}= \left\{
\partial_{1}l(x ,\theta), -\frac{\mu}{\sigma}\partial_{1}l(x ,\theta) + \frac{1}{2\sigma}\partial_{2}l(x
,\theta)\right\}$\\
We have $\left\{
   \begin{array}{ll}
     \partial_{\xi_{1}}l(x,\theta)= \frac{x - \mu}{\sigma^{2}}& \hbox{} \\
     \partial_{\xi_{2}}l(x,\theta) =  \frac{(x- \mu)^{2}}{2 \sigma^{4}} - \frac{\mu(x- \mu)}{\sigma^{3}} - \frac{1}{2\sigma^{2}} & \hbox{}
   \end{array}
 \right.
$
 \end{proof}

\begin{remark}
According to \cite{Shun}, the riemannian metric is given by
\begin{equation}\label{me2}
    g_{ij}(\theta) = <
\partial_{i}l(x, \theta),
\partial_{j}l(x, \theta)>  \textrm{where}, g_{ij}(\theta) = \mathbb{E}(\partial_{i}l(x,
\theta).
\partial_{j}l(x, \theta))=-\mathbb{E}(\partial_{i}
\partial_{j}l(x, \theta)).
\end{equation}
and $g^{ij}$ will be its reverse called stamp information of Fisher.
In the same, the metric one associated in the system from the $
\xi$- coordinates is defined by
\begin{equation}\label{def}
    g_{\alpha \beta}(\xi)
= < \partial_{\xi_{\alpha}},\partial_{\xi_{\beta}}>;\;   \alpha,
\beta \in \{1, 2 \}
\end{equation}
\end{remark}

\subsection{Affine connection.}
The collection of all $T_{p}S$ defined a $2n-$dimensional manifold
$T(S)$ called Tangent bundle of $S$. More explicitly, $T(S)$ is
completely defined by the projection mapping $\pi:TS\longrightarrow
S$ such that $\pi^{-1}(\{p\}) \cong\mathbb{R}^{n}$ and $\pi$ have
the usual local trivialization property. $A:S\longrightarrow TS$ is
a vector field on $S$ if $\pi\circ A=1_{S}$.
 The collection $\mathfrak{X}(S)$ of all vector fields on $S$ is an $C^{\infty}(S)-$module. An affine connection is
an $\mathbb{R}$-bilinear mapping\\
$ \nabla:\mathfrak{X}(S)\longrightarrow\mathfrak{X}(S):
(A,B)\mapsto\nabla_{A}B $
such that:\\
 $\nabla_{ A} (fB )= (Af) (B)  + f\nabla_{ A} (B),\;
 \nabla_{ fA} (B )= f\nabla_{ A} (B).$\\
In local coordinate
$\theta=\left(\theta_{1},\dots,\theta_{n}\right)$ of $S$ , $\nabla_{
\partial_{i}} (\partial_{j}
)=\sum\Gamma_{ij}^{k}\partial_{i}$ and $\Gamma_{ij}^{k}$ is
Christoffel coefficient given by: $ \Gamma_{ij}^{k}=
\frac{1}{2}g^{km}[\partial_{i}.g_{jm}+
\partial_{j}.g_{ik} - \partial_{k}.g_{ij}].
$
\subsection{Torsion tensor and curvature of affine connection.}
Let $\nabla$ be an affine-connection on $S$. A torsion tensor of
$\nabla$ is an $\mathbb{R}$-bilinear mapping
\begin{eqnarray*}T:\mathfrak{X}(S)\times
\mathfrak{X}(S)&\longrightarrow&\mathfrak{X}(S):(A,B)\longmapsto
T(A,B)=\nabla_{ A} (B)-\nabla_{ B} (A)-[A,B]\end{eqnarray*} In the
case of statistical manifold
 $S =
\left\{p_{ \theta}(x),\left.
                        \begin{array}{ll}
                         \theta\in \Theta & \hbox{} \\
                          x\in \mathcal{X} & \hbox{}
                        \end{array}
                      \right.
\right\}$. Amari \cite{Shun}, show that
 \begin{eqnarray}\label{e4}\mathbb{E}[\partial_{i}l(x , \theta)]&=&
0.\end{eqnarray} Where $l(x , \theta)$ is the $l-$representation of
$T_{\theta}$. In this case by taking into consideration the Amari
\cite{Shun}, the result affine connection  of $T_{\theta}$ on
statistical manifold $S = \left\{p_{ \theta}(x),\left.
                        \begin{array}{ll}
                         \theta\in \Theta & \hbox{} \\
                          x\in \mathcal{X} & \hbox{}
                        \end{array}
                      \right.
\right\}$ is define by
\begin{eqnarray}\label{e5}\Gamma_{ijk}:=
\mathbb{E}\left[\partial_{i}\partial_{j}l(x,\theta).\partial_{k}l(x,\theta)\right].\end{eqnarray}
The family $\Gamma_{ijk}$ are called $l-$representation of
coefficients of the connection $\nabla$ on $S$. With respect to  the
$l-$representation, $T_{ijk}(\theta)= \langle
T(\partial_{\theta_{i}},\partial_{j}),\partial_{k}\rangle,$ it is
the coefficient of $T$ with respect to
$\partial_{\theta}l(x,\theta)$. Since $T(\partial_{i},
\partial_{j}) = \nabla_{\partial_{i}}(\partial_{j}) -
\nabla_{\partial_{j}}(\partial_{i}) - (\partial_{i}\partial_{j}
-\partial_{j}\partial_{i})$, then $T_{ijk}(\theta)=
<\nabla_{\partial_{i}}(\partial_{j}) -
\nabla_{\partial_{j}}(\partial_{i}) - (\partial_{i}\partial_{j}
-\partial_{j}\partial_{i}),
\partial_{k}>.$
We obtain the following expression: $T_{ijk}(\theta)=
<\nabla_{\partial_{i}}(\partial_{j}),
\partial_{k}> -
<\nabla_{\partial_{j}}(\partial_{i}),
\partial_{k}>.$ Therefore,
\begin{eqnarray}\label{e6}
T_{ijk}(\theta)=\Gamma_{ijk}(\theta ) - \Gamma_{jik}(\theta).
\end{eqnarray}

The curvature tensor of affine connection $\nabla$ on a $n$-manifold
$S$ is a $\mathbb{R}-$trilinear mapping $R:\mathfrak{X}(S)\times
\mathfrak{X}(S)\times \mathfrak{X}(S)\longrightarrow\mathfrak{X}(S):
(X,Y,Z)\longmapsto R(X,Y)Z$\\ such that
$R(X,Y)Z=[\nabla_{X},\nabla_{Y}]Z-\nabla_{[X,Y]}Z,$ with
$[\nabla_{X},\nabla_{Y}]Z=\nabla_{X}\nabla_{Y}Z
-\nabla_{Y}\nabla_{X}Z.$
 In a locale coordinates $\theta=\left(\theta_{1}, \dots,\theta_{n}\right)$ of $S$, we have
$R(\partial_{i},\partial_{j}).\partial_{k}=\sum_{m=1}^{n}R_{ijk}^{m}\partial_{m}$\\
where $R_{ijk}^{m}=\frac{\partial
\Gamma_{kl}^{m}}{\partial_{j}}-\frac{\partial
\Gamma_{jl}^{m}}{\partial_{k}}+\Gamma_{js}^{m}\Gamma_{kl}^{s}-\Gamma_{ks}^{m}\Gamma_{jl}^{s}.$
Taking in consideration the $l-$representation of $T_{\theta}$ on
statistical manifold
 $S =
\left\{p_{ \theta}(x),\left.
                        \begin{array}{ll}
                         \theta\in \Theta & \hbox{} \\
                          x\in \mathcal{X} & \hbox{}
                        \end{array}
                      \right.
\right\}$. The family
\begin{eqnarray}\label{e7}
R_{ijkm}:=\langle
R(\partial_{i},\partial_{j},\partial_{k}),\partial_{m}\rangle\end{eqnarray}
is the $l-$ representation of Riemann-Christoffel curvature tensor
of the affine connection $\nabla$ on $S.$ It follow From (\ref{e7})
that $R_{ijkm}$ the following expression:
\begin{eqnarray}\label{e8}
 R_{ijkm}=\left(\partial_{i}\Gamma^{s}_{jk} -
\partial_{j}\Gamma^{s}_{ik} \right)g_{sm}+\left(\Gamma_{irm}\Gamma^{r}_{jk}- \Gamma_{jrm}\Gamma^{r}_{ik}\right).\end{eqnarray}
where $g_{sm}$ are coefficients of the Riemannian metric $g$ on $S$.
A space with an affine connection is said to be flat when the
Riemann-Christoffel curvature and torsion vanishes identically, i.e;
$R_{ijkm}(\theta)=0$ holds for any $\theta$. In the coordinate
system $\theta=\left(\theta_{1},\dots,\theta_{n}\right)$, the scalar
curvature $\mathcal{K }$ of affine connection $\nabla$ is defined by
\begin{eqnarray}\label{e9}
 \mathcal{K}=\frac{1}{n(n-1)}R_{ijkm}g^{im}g^{jk}.\end{eqnarray}
So we have the  scalar curve in the frame of reference $ \xi$ the
noted number $ \mathcal{K}$  define by
\begin{equation}\label{def1}\mathcal{K}= \frac{1}{2}R_{\alpha \beta \gamma
\eta}g^{\alpha \eta} g^{\beta \gamma}.\end{equation}

According to \cite{Shun}, and using \ref{me2}, we have
 \begin{eqnarray*}g_{11}(\theta)&=&  - \mathbb{E}\left[ \partial_{1}\partial_{1}l(x, \theta)\right]\\
&=& - \mathbb{E}\left[ -\frac{1}{\sigma^{2}}\right]\\
 &=& \frac{1}{\sigma^{2}}.
\end{eqnarray*}

\begin{eqnarray*} g_{22}(\theta)&=& - \mathbb{E}\left[ \partial_{2}\partial_{2}l(x, \theta)\right]\\
&=& - \mathbb{E}\left[ \frac{1}{\sigma^{2}} - \frac{3(x- \mu)^{2}}{\sigma^{4}}\right]\\
&=& -\mathbb{E}\left(\frac{1}{\sigma^{2}}\right) -\mathbb{E}\left[- \frac{3(x- \mu)^{2}}{\sigma^{4}}\right]\\
&=& -\frac{1}{\sigma^{2}} + 3\mathbb{E}\left[ \frac{(x- \mu)^{2}}{\sigma^{4}}\right]  \\
&=& -\frac{1}{\sigma^{2}} + 3\mathbb{E}\left[ \partial_{1}l(x, \theta).\partial_{1}l(x, \theta)\right] \\
&=& -\frac{1}{\sigma^{2}} - 3\mathbb{E}\left[ \partial_{1}\partial_{1}l(x, \theta)\right]\\
&=&-\frac{1}{\sigma^{2}} - 3\mathbb{E}\left[ -\frac{1}{\sigma^{2}}\right]\\
&=&-\frac{1}{\sigma^{2}} +\frac{3}{\sigma^{2}}\\
&=& \frac{2}{\sigma^{2}}
\end{eqnarray*}

\begin{eqnarray*}
g_{12}(\theta)&=& - \mathbb{E}\left[ \partial_{1}\partial_{2}l(x, \theta)\right]\\
&=& - \mathbb{E}\left[ \frac{2}{\sigma^{3}}(x- \mu)\right]\\
&=& -\frac{2}{\sigma^{\sigma}}\mathbb{E}\left[\frac{(x - \mu)}{\sigma^{2}}\right] \\
&=& -\frac{2}{\sigma}\mathbb{E}\left[\partial_{1}l(x, \theta)\right]\\
&=& 0
\end{eqnarray*}

\begin{eqnarray*}
g_{21}(\theta)&=& - \mathbb{E}\left[ \partial_{2}\partial_{1}l(x, \theta)\right]\\
&=& - \mathbb{E}\left[ -\frac{2}{\sigma^{3}}(x- \mu)\right]\\
&=& \frac{2}{\sigma^{3}}\mathbb{E}\left[(x - \mu)\right]\\
&=& \frac{2}{\sigma}\mathbb{E}\left[\frac{(x - \mu)}{\sigma^{2}}\right]\\
&=& \frac{2}{\sigma}\mathbb{E}\left[\partial_{1}l(x, \theta)\right]\\
&=& 0
\end{eqnarray*}

So, the metric one associated for a manifold of gaussian model $S$
is given by
\begin{equation}\label{b6}
    g_{ij}(\theta) = \left(
\begin{array}{cc}
\frac{1}{\sigma^{2}} & 0 \\
0 & \frac{2}{\sigma^{2}} \\
\end{array}
\right)
\end{equation}

\section{Riemannian metric in $\xi$-coordinates system.}

\begin{theorem}\label{b7}

 Let $\xi = \left(\mu,  \mu^{2} + \sigma^{2}\right)$ be the dual coordinate system
 associated with the Gaussian statistical manifold.
  Let $\left( g_{ij}\right)_{1\leq
i,j\leq2}$ denote the Riemannian
   metric expressed in the natural coordinate system $\theta = (\mu, \sigma)$,
   as defined by equation (\ref{b6}). Let $(\overline{\mathcal{B}}^{i}_{\alpha} )$ denote the inverse of the
   Jacobian matrix of the coordinate transformation from $\theta$ to $\xi$, given explicitly by:
$(\overline{\mathcal{B}}^{i}_{\alpha} ) = \left(
                                                  \begin{array}{cc}
                                                   1 & 0 \\
                                                  \frac{- \mu}{\sigma}   & \frac{1}{2\sigma} \\
                                                  \end{array}
                                                \right)$
 Then, the components of the Riemannian metric $g_{\alpha\beta}(\xi)$ in the $\xi$-coordinate system are given by:
\begin{equation*}
    g_{\alpha
\beta}(\xi) =
\frac{1}{\sigma^{2}}\bar{B_{\alpha}}^{1}\bar{B_{\beta}}^{1} +
\frac{2}{\sigma^{2}}\bar{B_{\alpha}}^{2}\bar{B_{\beta}}^{2}
\end{equation*}
$\alpha,\; \beta\in\left\{1, 2\right\}$.
\end{theorem}

\begin{proof}
Let $\theta = (\mu, \sigma)$ denote the natural coordinate system
and $\xi = (\mu, \mu^2 + \sigma^2)$ the dual coordinate system on
the Gaussian manifold. From the coordinate transformation, the
Jacobian matrix of the map $\theta \mapsto \xi$ is given by:
$$
J =
\begin{pmatrix}
\frac{\partial \xi_{1}}{  \partial \mu }& \frac{\partial \xi_{1}}{  \partial \sigma } \\
\frac{\partial \xi_{2}}{  \partial \mu } & \frac{\partial \xi_{2}}{
\partial \sigma }
\end{pmatrix}
=
\begin{pmatrix}
1 & 0 \\
2\mu & 2\sigma
\end{pmatrix},
$$
whose inverse is:
$$
J^{-1} = (\overline{\mathcal{B}}^{i}_{\alpha} ) =
\begin{pmatrix}
1 & 0 \\
-\frac{\mu }{ \sigma} & \frac{1}{ 2\sigma}
\end{pmatrix}.
$$
Let $\{\partial_1, \partial_2\} = \left\{ \frac{\partial}{\partial
\mu}, \frac{\partial}{\partial \sigma} \right\}$ be the basis
associated with the $\theta$-coordinates, and $\{ \partial_{\xi_1},
\partial_{\xi_2} \}$ the basis in the $\xi$-coordinate system. Using
the chain rule and the inverse Jacobian, the basis transformation is
\begin{equation*}\partial_{\xi_{\alpha}} = \sum\limits_{i
=1}^{2}\bar{B_{\alpha}}^{i}\partial_{i}=\bar{B_{\alpha}}^{1}\partial_{1}+
\bar{B_{\alpha}}^{2}\partial_{2},\; and\;
\partial_{\xi_{\beta}} = \sum\limits_{j
=1}^{2}\bar{B_{\beta}}^{j}\partial_{j}=
\bar{B_{\beta}}^{1}\partial_{1}+ \bar{B_{\beta}}^{2}\partial_{2}.
\end{equation*}
 Then, the components of
the Riemannian metric in the $\xi$-coordinates are given by the
inner product
 \begin{equation*}g_{\alpha \beta}(\xi) = <\partial_{\xi_{\alpha}},
\partial_{\xi_{\beta}}>,\;\alpha,\; \beta\in\left\{1, 2\right\}.\end{equation*}
 or $\bar{B_{\alpha}}^{i} = \left( \frac{\partial
\theta_{i}}{\partial \xi_{\alpha}}\right)$
 We have
 \begin{eqnarray*}
  g_{\alpha \beta}(\xi)&=& <\bar{B_{\alpha}}^{1}\partial_{1}+ \bar{B_{\alpha}}^{2}\partial_{2} ,  \bar{B_{\beta}}^{1}\partial_{1}+ \bar{B_{\beta}}^{2}\partial_{2} >\\
&=& <\bar{B_{\alpha}}^{1}\partial_{1} ,
\bar{B_{\beta}}^{1}\partial_{1} > +
<\bar{B_{\alpha}}^{1}\partial_{1} , \bar{B_{\beta}}^{2}\partial_{2}>
 + <\bar{B_{\alpha}}^{2}\partial_{2} ,\bar{B_{\beta}}^{1}\partial_{1}> + <\bar{B_{\alpha}}^{2}\partial_{2} , \bar{B_{\beta}}^{2}\partial_{2}>\\
&=& \bar{B_{\alpha}}^{1}\bar{B_{\beta}}^{1}g_{11}(\theta) +
\bar{B_{\alpha}}^{2}\bar{B_{\beta}}^{2}g_{22}(\theta)
\end{eqnarray*} $\forall \alpha,\; \beta\in\left\{1, 2\right\}$.
So,\begin{equation*}
    g_{\alpha
\beta}(\xi) =
\frac{1}{\sigma^{2}}\bar{B_{\alpha}}^{1}\bar{B_{\beta}}^{1} +
\frac{2}{\sigma^{2}}\bar{B_{\alpha}}^{2}\bar{B_{\beta}}^{2}
\end{equation*}
\end{proof}

While being  (\ref{b6}) and (\ref{b7}) there is the proof of
following theorem
\begin{theorem}\label{b5}
 Let $\xi = \left(\mu,  \mu^{2} + \sigma^{2}\right)$ be the dual coordinate
  system associated with the natural coordinates $\theta = \left(\mu, \sigma\right)$
  on the Gaussian statistical manifold. Then, the Riemannian metric $G_d$
  in the $\xi$-coordinate system is given by:
\begin{equation*}
    G_{d} = \frac{1}{\sigma^{4}}\left(
\begin{array}{cc}
\sigma^{2} + 2 \mu^{2} & - \mu\\
- \mu & \frac{1}{2}
\end{array}
\right),
\end{equation*}
 The determinant of this metric tensor is:
\begin{equation*}
    Det G_{d} = \frac{1}{2\sigma^{6}}
\end{equation*}
 The inverse of the Fisher information matrix in the $\xi$-coordinate system, denoted $G_d^{-1}$, is given by:
\begin{equation*}
    G_{d}^{-1} = \left(
\begin{array}{cc}
\sigma^{2}   & 2 \mu \sigma^{2}\\
2 \mu \sigma^{2} & 2 \sigma^{4} + 4 \mu^{2}\sigma^{2}
\end{array}
\right)
\end{equation*}

\end{theorem}

\begin{proof}
Using theorem \ref{b7}, we have \begin{equation*}g_{11} (\xi) =
\bar{B_{1}}^{1}\bar{B_{1}}^{1} g_{11}(\theta)
+\bar{B_{1}}^{2}\bar{B_{1}}^{2} g_{22}(\theta)\end{equation*}
\begin{equation*}g_{11}(\xi) = \left(\frac{\partial
\theta_{1}}{\partial \xi_{1}}\right)\left(\frac{1}{\sigma^{2}}
\right) + \left(\frac{\partial \theta_{2}}{\partial
\xi_{1}}\right)\left(\frac{2}{\sigma^{2}} \right)\end{equation*}
because $\theta_{1} = \mu, \; \xi_{1} = \mu $

\begin{eqnarray*}
g_{11} (\xi) &=& \frac{1}{\sigma^{2}} + \frac{2 \mu^{2}}{\sigma^{4}}\\
&=& \frac{1}{\sigma^{2}} \left(\frac{\sigma^{2} + 2
\mu^{2}}{\sigma^{2}}\right) \\
&=& \frac{1}{\sigma^{4}} \left(\sigma^{2} + 2 \mu^{2}\right)
\end{eqnarray*}
then we have
\begin{eqnarray*}
g_{22}(\xi) &=& \bar{B_{2}}^{1}\bar{B_{2}}^{1} g_{11}(\theta) + \bar{B_{2}}^{2}\bar{B_{2}}^{2} g_{22}(\theta)\\
g_{22}(\xi) &=& \left(\bar{B_{2}}^{1}\right)^{2}  g_{11}(\theta) +\left(\bar{B_{2}}^{2}\right)^{2}  g_{22}(\theta)\\
g_{22}(\xi) &=& \left(\bar{B_{2}}^{1}\right)^{2}
\left(\frac{1}{\sigma^{2}}\right) +\left(\bar{B_{2}}^{2}\right)^{2}
\left(\frac{2}{\sigma^{2}}\right)\\
g_{22}(\xi) &=&\left(\frac{1}{2\sigma}\right)^{2}\left(\frac{2}{\sigma^{2}}\right)\\
&=& \frac{2}{\sigma^{4}} \\
i.e. \; g_{22}(\xi)&= \frac{1}{2\sigma^{4}}
\end{eqnarray*}
\begin{eqnarray*}
g_{12}(\xi) &=& \bar{B_{1}}^{1}\bar{B_{2}}^{1}g_{11}(\theta) +
\bar{B_{1}}^{2}\bar{B_{2}}^{2}g_{22}(\theta)\\
&=& \frac{-\mu}{\sigma^{4}}
\end{eqnarray*}

\begin{eqnarray*}
g_{21}(\xi)&=& \bar{B_{2}}^{1}\bar{B_{1}}^{1}g_{11}(\theta) + \bar{B_{2}}^{1}\bar{B_{1}}^{2}g_{12}(\theta)+ \bar{B_{2}}^{2}\bar{B_{1}}^{1}g_{21}(\theta)+ \bar{B_{2}}^{2}\bar{B_{1}}^{2}g_{22}(\theta) \\
&=& \bar{B_{2}}^{1}\bar{B_{1}}^{1}g_{11}(\theta) + \bar{B_{2}}^{2}\bar{B_{1}}^{2}g_{22}(\theta)\\
&=& \frac{-\mu}{\sigma^{4}}
\end{eqnarray*}we have the following metric
\begin{equation*}
    G_{d} = \frac{1}{\sigma^{4}}\left(
\begin{array}{cc}
\sigma^{2} + 2 \mu^{2} & - \mu\\
- \mu & \frac{1}{2}
\end{array}
\right),
\end{equation*}
 and
\begin{equation*}
    Det G_{d} = \frac{1}{2\sigma^{6}}
\end{equation*}
So, we have the result
\end{proof}

\section{Torsion and curvature tensor in $\xi$-coordinates system.}

\begin{theorem}\label{b50}
Let $\nabla^{(1)}$ be the $\alpha$-connection
 of Amari with $\alpha = 1$. Then,
 the torsion tensor $T^\gamma_{\alpha\beta}(\xi)$ in the dual coordinate
  system $\xi = \left(\mu,  \mu^{2} + \sigma^{2}\right)$ is non-zero. This is in contrast
   to the natural coordinate system $\theta$, where the torsion vanishes
   for the Levi-Civita connection. The presence of torsion in $\xi$-coordinates
   reflects the fact that duality does not preserve symmetry of the connection.
\end{theorem}

\begin{proof}According to Amari's \cite{Shun}, we have Christoffel's coefficient in
$\theta$-coordinate system on gaussian statistical manifold
\[\Gamma^{(1)}_{111}(\theta) = \Gamma^{(1)}_{112}(\theta) =
\Gamma^{(1)}_{221}(\theta) =
\Gamma^{(1)}_{122}(\theta)=\Gamma^{(1)}_{212}(\theta) = 0  \]
\[\Gamma^{(1)}_{121}(\theta) = \Gamma^{(1)}_{211}(\theta) = -\frac{2}{\sigma^{3}}, \; \Gamma^{(1)}_{222}(\theta) =  -\frac{6}{\sigma^{3}}. \]Calculation of the coefficients of
christoffel's in the system of coordinates $ \xi$.
\begin{eqnarray*}
  \Gamma^{1}_{111}(\xi) &=& \sum\limits_{i= 1}^{2} \sum\limits_{j= 1}^{2}\sum\limits_{k= 1}^{2} \bar{B_{1}}^{i}\bar{B_{1}}^{j}\bar{B_{1}}^{k}\Gamma_{ijk}(\theta) \\
   &=& \bar{B_{1}}^{1}\bar{B_{1}}^{1}\bar{B_{1}}^{1}\Gamma_{111}(\theta)+\bar{B_{1}}^{1}\bar{B_{1}}^{1}\bar{B_{1}}^{2}\Gamma_{112}(\theta) +\nonumber\\
&&\bar{B_{1}}^{1}\bar{B_{1}}^{2}\bar{B_{1}}^{1}\Gamma_{121}(\theta)+\bar{B_{1}}^{1}\bar{B_{1}}^{2}\bar{B_{1}}^{2}\Gamma_{122}(\theta) + \nonumber\\
&&+ \bar{B_{1}}^{2}\bar{B_{1}}^{1}\bar{B_{1}}^{1}\Gamma_{211}(\theta)+\bar{B_{1}}^{2}\bar{B_{1}}^{1}\bar{B_{1}}^{2}\Gamma_{212}(\theta) +\nonumber\\
&&\bar{B_{1}}^{2}\bar{B_{1}}^{2}\bar{B_{1}}^{1}\Gamma_{221}(\theta)+\bar{B_{1}}^{2}\bar{B_{1}}^{2}\bar{B_{1}}^{2}\Gamma_{222}(\theta) \\
  &=& -\frac{\mu}{\sigma}.(\frac{-2}{\sigma^{3}}) - \frac{\mu}{\sigma}.(\frac{-2}{\sigma^{3}})  + \frac{-\mu^{3}}{\sigma^{3}}.(\frac{-6}{\sigma^{3}})\\
 &=& \frac{4\mu\sigma^{2} + 6\mu^{3}}{\sigma^{6}}\\
 \Gamma_{111}(\xi)&=& \frac{4\mu\sigma^{2} +
6\mu^{3}}{\sigma^{6}}
\end{eqnarray*}

By analogy with calculation preceding one a:
\begin{eqnarray*}
\Gamma_{112}(\xi)&=& \sum\limits_{i= 1}^{2} \sum\limits_{j= 1}^{2}\sum\limits_{k= 1}^{2} \bar{B_{1}}^{i}\bar{B_{1}}^{j}\bar{B_{2}}^{k}\Gamma_{ijk}(\theta) \\
&=& \bar{B_{1}}^{1}\bar{B_{1}}^{1}\bar{B_{2}}^{1}\Gamma_{111}(\theta)+\bar{B_{1}}^{1}\bar{B_{1}}^{1}\bar{B_{2}}^{2}\Gamma_{112}(\theta) +\nonumber\\
&&\bar{B_{1}}^{1}\bar{B_{1}}^{2}\bar{B_{2}}^{1}\Gamma_{121}(\theta)+\bar{B_{1}}^{1}\bar{B_{1}}^{2}\bar{B_{2}}^{2}\Gamma_{122}(\theta) + \nonumber\\
&& \bar{B_{1}}^{2}\bar{B_{1}}^{1}\bar{B_{2}}^{1}\Gamma_{211}(\theta)+\bar{B_{1}}^{2}\bar{B_{1}}^{1}\bar{B_{2}}^{2}\Gamma_{212}(\theta) +\nonumber\\
&&\bar{B_{1}}^{2}\bar{B_{1}}^{2}\bar{B_{2}}^{1}\Gamma_{221}(\theta)+\bar{B_{1}}^{2}\bar{B_{1}}^{2}\bar{B_{2}}^{2}\Gamma_{222}(\theta) \\
 &=&   \bar{B_{1}}^{2}\bar{B_{1}}^{2}\bar{B_{2}}^{2}\Gamma_{222}(\theta)\\
&=&  \frac{\mu^{2}}{\sigma^{2}}.(\frac{1}{2\sigma})(\frac{-6}{\sigma^{3}})\\
 \Gamma_{112}(\xi)&=&   - \frac{3\mu^{2}}{\sigma^{6}}
\end{eqnarray*}
In the same way,
\begin{eqnarray*}
\Gamma_{221}(\xi)&=& \sum\limits_{i= 1}^{2} \sum\limits_{j= 1}^{2}\sum\limits_{k= 1}^{2} \bar{B_{2}}^{i}\bar{B_{2}}^{j}\bar{B_{1}}^{k}\Gamma_{ijk}(\theta) \\
 &=& \bar{B_{2}}^{1}\bar{B_{2}}^{1}\bar{B_{1}}^{1}\Gamma_{111}(\theta)+\bar{B_{2}}^{1}\bar{B_{2}}^{1}\bar{B_{1}}^{2}\Gamma_{112}(\theta) +\nonumber\\
&&\bar{B_{2}}^{1}\bar{B_{2}}^{2}\bar{B_{1}}^{1}\Gamma_{121}(\theta)+\bar{B_{2}}^{1}\bar{B_{2}}^{2}\bar{B_{1}}^{2}\Gamma_{122}(\theta) + \nonumber \\
&&\bar{B_{2}}^{2}\bar{B_{2}}^{1}\bar{B_{1}}^{1}\Gamma_{211}(\theta)+\bar{B_{2}}^{2}\bar{B_{2}}^{1}\bar{B_{1}}^{2}\Gamma_{212}(\theta) +\nonumber\\
&&\bar{B_{2}}^{2}\bar{B_{2}}^{2}\bar{B_{1}}^{1}\Gamma_{221}(\theta)+\bar{B_{2}}^{2}\bar{B_{2}}^{2}\bar{B_{1}}^{2}\Gamma_{222}(\theta) \\
 &=&\bar{B_{2}}^{2}\bar{B_{2}}^{2}\bar{B_{1}}^{2}\Gamma_{222}(\theta)\\
 \Gamma_{221}(\xi)&=& \frac{3\mu}{2\sigma^{6}}
\end{eqnarray*}

\begin{eqnarray*}
\Gamma_{212}(\xi)&=& \sum\limits_{i= 1}^{2} \sum\limits_{j= 1}^{2}\sum\limits_{k= 1}^{2} \bar{B_{2}}^{i}\bar{B_{1}}^{j}\bar{B_{2}}^{k}\Gamma_{ijk}(\theta) \\
 &=& \bar{B_{2}}^{1}\bar{B_{1}}^{1}\bar{B_{2}}^{1}\Gamma_{111}(\theta)+\bar{B_{2}}^{1}\bar{B_{1}}^{1}\bar{B_{2}}^{2}\Gamma_{112}(\theta) +\nonumber\\
&&\bar{B_{2}}^{1}\bar{B_{1}}^{2}\bar{B_{2}}^{1}\Gamma_{121}(\theta)+\bar{B_{2}}^{1}\bar{B_{1}}^{2}\bar{B_{2}}^{2}\Gamma_{122}(\theta) + \nonumber\\
&&\bar{B_{2}}^{2}\bar{B_{1}}^{1}\bar{B_{2}}^{1}\Gamma_{211}(\theta)+\bar{B_{2}}^{2}\bar{B_{1}}^{1}\bar{B_{2}}^{2}\Gamma_{212}(\theta) +\nonumber\\
&&\bar{B_{2}}^{2}\bar{B_{1}}^{2}\bar{B_{2}}^{1}\Gamma_{221}(\theta)+\bar{B_{2}}^{2}\bar{B_{1}}^{2}\bar{B_{2}}^{2}\Gamma_{222}(\theta) \\
 \Gamma_{212}(\xi)&=&  \frac{3\mu}{2\sigma^{6}}
\end{eqnarray*}
In the same way in a similar way,
\begin{eqnarray*}
\Gamma_{121}(\xi)&=& \sum\limits_{i= 1}^{2} \sum\limits_{j= 1}^{2}\sum\limits_{k= 1}^{2} \bar{B_{1}}^{i}\bar{B_{2}}^{j}\bar{B_{1}}^{k}\Gamma_{ijk}(\theta) \\
  &=& \bar{B_{1}}^{1}\bar{B_{2}}^{1}\bar{B_{1}}^{1}\Gamma_{111}(\theta)+\bar{B_{1}}^{1}\bar{B_{2}}^{1}\bar{B_{1}}^{2}\Gamma_{112}(\theta) +\nonumber\\
&&\bar{B_{1}}^{1}\bar{B_{2}}^{2}\bar{B_{1}}^{1}\Gamma_{121}(\theta)+\bar{B_{1}}^{1}\bar{B_{2}}^{2}\bar{B_{1}}^{2}\Gamma_{122}(\theta) + \nonumber \\
&&\bar{B_{1}}^{2}\bar{B_{2}}^{1}\bar{B_{1}}^{1}\Gamma_{211}(\theta)+\bar{B_{1}}^{2}\bar{B_{2}}^{1}\bar{B_{1}}^{2}\Gamma_{212}(\theta) +\nonumber\\
&&\bar{B_{1}}^{2}\bar{B_{2}}^{2}\bar{B_{1}}^{1}\Gamma_{221}(\theta)+\bar{B_{1}}^{2}\bar{B_{2}}^{2}\bar{B_{1}}^{2}\Gamma_{222}(\theta) \\
 &=& \bar{B_{1}}^{1}\bar{B_{2}}^{2}\bar{B_{1}}^{1}\Gamma_{121}(\theta) + \bar{B_{1}}^{2}\bar{B_{2}}^{2}\bar{B_{1}}^{2}\Gamma_{222}(\theta)\\
\ \Gamma_{121}(\xi)&=& -\frac{1}{\sigma^{4}} -
\frac{3\mu^{2}}{\sigma^{6}}
\end{eqnarray*}

In the same way one obtains

\begin{eqnarray*}
\Gamma_{211}(\xi)&=& \sum\limits_{i= 1}^{2} \sum\limits_{j= 1}^{2}\sum\limits_{k= 1}^{2} \bar{B_{2}}^{i}\bar{B_{1}}^{j}\bar{B_{1}}^{k}\Gamma_{ijk}(\theta) \\
  &=& \bar{B_{2}}^{1}\bar{B_{1}}^{1}\bar{B_{1}}^{1}\Gamma_{111}(\theta)+\bar{B_{2}}^{1}\bar{B_{1}}^{1}\bar{B_{1}}^{2}\Gamma_{112}(\theta) +\nonumber\\
&&\bar{B_{2}}^{1}\bar{B_{1}}^{2}\bar{B_{1}}^{1}\Gamma_{121}(\theta)+\bar{B_{2}}^{1}\bar{B_{1}}^{2}\bar{B_{1}}^{2}\Gamma_{122}(\theta) +\nonumber \\
&&\bar{B_{2}}^{2}\bar{B_{1}}^{1}\bar{B_{1}}^{1}\Gamma_{211}(\theta)+\bar{B_{2}}^{2}\bar{B_{1}}^{1}\bar{B_{1}}^{2}\Gamma_{212}(\theta) +\nonumber\\
&&\bar{B_{2}}^{2}\bar{B_{1}}^{2}\bar{B_{1}}^{1}\Gamma_{221}(\theta)+\bar{B_{2}}^{2}\bar{B_{1}}^{2}\bar{B_{1}}^{2}\Gamma_{222}(\theta) \\
 &=& \bar{B_{2}}^{2}\bar{B_{1}}^{1}\bar{B_{1}}^{1}\Gamma_{211}(\theta)  + \bar{B_{2}}^{2}\bar{B_{1}}^{2}\bar{B_{1}}^{2}\Gamma_{222}(\theta)\\
\Gamma_{211}(\xi)&=& - \frac{1}{\sigma^{4}} -
\frac{3\mu^{2}}{\sigma^{6}}
\end{eqnarray*}

\begin{eqnarray*}
\Gamma_{122}(\xi)&=& \sum\limits_{i= 1}^{2} \sum\limits_{j= 1}^{2}\sum\limits_{k= 1}^{2} \bar{B_{1}}^{i}\bar{B_{2}}^{j}\bar{B_{2}}^{k}\Gamma_{ijk}(\theta) \\
 &=& \bar{B_{1}}^{1}\bar{B_{2}}^{1}\bar{B_{2}}^{1}\Gamma_{111}(\theta)+\bar{B_{1}}^{1}\bar{B_{2}}^{1}\bar{B_{2}}^{2}\Gamma_{112}(\theta) +\nonumber\\
&&\bar{B_{1}}^{1}\bar{B_{2}}^{2}\bar{B_{2}}^{1}\Gamma_{121}(\theta)+\bar{B_{1}}^{1}\bar{B_{2}}^{2}\bar{B_{2}}^{2}\Gamma_{122}(\theta) +\nonumber \\
&&\bar{B_{1}}^{2}\bar{B_{2}}^{1}\bar{B_{2}}^{1}\Gamma_{211}(\theta)+\bar{B_{1}}^{2}\bar{B_{2}}^{1}\bar{B_{2}}^{2}\Gamma_{212}(\theta) +\nonumber\\
&&\bar{B_{1}}^{2}\bar{B_{2}}^{2}\bar{B_{2}}^{1}\Gamma_{221}(\theta)+\bar{B_{1}}^{2}\bar{B_{2}}^{2}\bar{B_{2}}^{2}\Gamma_{222}(\theta) \\
 &=&  \bar{B_{1}}^{2}\bar{B_{2}}^{2}\bar{B_{2}}^{2}\Gamma_{222}(\theta) \\
 \Gamma_{122}(\xi)&=& \frac{3\mu}{2\sigma^{6}}
\end{eqnarray*}

\begin{eqnarray*}
\Gamma_{222}(\xi)&=& \sum\limits_{i= 1}^{2} \sum\limits_{j= 1}^{2}\sum\limits_{k= 1}^{2} \bar{B_{2}}^{i}\bar{B_{2}}^{j}\bar{B_{2}}^{k}\Gamma_{ijk}(\theta) \\
  &=& \bar{B_{2}}^{1}\bar{B_{2}}^{1}\bar{B_{2}}^{1}\Gamma_{111}(\theta)+\bar{B_{2}}^{1}\bar{B_{2}}^{1}\bar{B_{2}}^{2}\Gamma_{112}(\theta) +\nonumber\\
&&\bar{B_{2}}^{1}\bar{B_{2}}^{2}\bar{B_{2}}^{1}\Gamma_{121}(\theta)+\bar{B_{2}}^{1}\bar{B_{2}}^{2}\bar{B_{2}}^{2}\Gamma_{122}(\theta) +\nonumber \\
&&\bar{B_{2}}^{2}\bar{B_{2}}^{1}\bar{B_{2}}^{1}\Gamma_{211}(\theta)+\bar{B_{2}}^{2}\bar{B_{2}}^{1}\bar{B_{2}}^{2}\Gamma_{212}(\theta) +\nonumber\\
&&\bar{B_{2}}^{2}\bar{B_{2}}^{2}\bar{B_{2}}^{1}\Gamma_{221}(\theta)+\bar{B_{2}}^{2}\bar{B_{2}}^{2}\bar{B_{2}}^{2}\Gamma_{222}(\theta) \\
 &=&\bar{B_{2}}^{2}\bar{B_{2}}^{2}\bar{B_{2}}^{1}\Gamma_{221}(\theta)+ \bar{B_{2}}^{2}\bar{B_{2}}^{1}\bar{B_{2}}^{2}\Gamma_{212}(\theta)+ \bar{B_{2}}^{2}\bar{B_{2}}^{2}\bar{B_{2}}^{2}\Gamma_{222}(\theta) \\
&=& \bar{B_{2}}^{2}\bar{B_{2}}^{2}\bar{B_{2}}^{2}\Gamma_{222}(\theta)\\
\Gamma_{222}(\xi) &=& -\frac{3}{4\sigma^{6}}
\end{eqnarray*}
Using \ref{e6}, we have the result.
\end{proof}

\begin{theorem}
 Let $\xi = \left(\mu, \mu^2 + \sigma^2\right)$ be
  the dual coordinate system associated
   with the Gaussian statistical manifold.
    Then, the components of the Riemann curvature
     tensor  of the affine connection $\nabla^{(1)}$ in these coordinates include:
\begin{eqnarray*}R_{1111}(\xi) &=& R_{2222}(\xi)= R_{1122}(\xi)= R_{1112}(\xi)= R_{2121}(\xi)= R_{2211}(\xi)=0\\
&& R_{2221}(\xi)= R_{2212}(\xi)= R_{2122}(\xi)=
R_{1121}(\xi)=R_{1211}(\xi) =0\\
R_{1221}(\xi)&=&\frac{\mu}{\sigma^{6}},\;
R_{2112}(\xi)=\frac{6\sigma^{8} + 3\mu\sigma^{6} + 6
\mu^{2}\sigma^{6}+ 6 \mu^{2}}{4 \sigma^{14}},\;
R_{1212}(\xi) =\frac{1}{2\sigma^{6}}\\
R_{2111}(\xi) &=&\frac{-2\sigma^{8} + 2 \mu^{2}\sigma^{6} + 46
\mu^{2}\sigma^{8} + 24 \mu^{4}\sigma^{6} + 12 \mu^{2}\sigma^{2} -
4\mu \sigma^{2} + 6\sigma^{10} + 12\mu^{4} - 3\mu^{2} + 9 \mu
\sigma^{8} + 18 \mu^{3}\sigma^{6}}{2\sigma^{14}}
\end{eqnarray*}
and the scalar curve is given by
\begin{eqnarray*}\mathcal{K}&=& \frac{16\mu^{2}\sigma^{8} + 12\mu^{4} + 6\sigma^{10} - \mu\sigma^{8} - 2 \mu^{2}\sigma^{2} +10\mu^{3}\sigma^{6} +92 \mu^{3}\sigma^{8} +48\mu^{5}\sigma^{6}}{4\sigma^{10}}\\
&& + \frac{12\mu\sigma^{10}+24\mu^{5}-6\mu^{3}}{4\sigma^{10}}
\end{eqnarray*}
\end{theorem}

\begin{proof}
According to Amari \cite{Shun}, we have
\begin{equation*}
    \Gamma_{\alpha \beta}^{\gamma}(\xi) = \frac{1}{2}g^{\gamma
\lambda}(\xi)\left(\frac{\partial g_{\beta \lambda}(\xi)}{\partial
\xi_{\alpha}} + \frac{\partial g_{\gamma \alpha}(\xi)}{\partial
\xi_{\beta}}- \frac{\partial g_{\alpha \beta}(\xi)}{\partial
\xi_{\gamma}}\right)
\end{equation*} with $\alpha,\; \beta,\;\gamma,\;\lambda\in\left\{1,2\right\}$.  So, we have
the coefficients following
\begin{eqnarray*}
\Gamma_{11}^{1}(\xi)&=& \frac{1}{2}g^{11}(\xi)\left(\frac{\partial
g_{11}(\xi)}{\partial \xi_{1}} + \frac{\partial
g_{11}(\xi)}{\partial \xi_{1}}
 - \frac{\partial g_{11}(\xi)}{\partial \xi_{1}}\right)+ \frac{1}{2}g^{12}(\xi)\frac{\partial g_{12}(\xi)}{\partial \xi_{1}}\\
\Gamma_{11}^{1}(\xi) &=& \frac{\mu}{\sigma^{2}}.
\end{eqnarray*}

\begin{eqnarray*}
\Gamma_{22}^{2}(\xi)&=& \frac{1}{2}g^{21}(\xi)\frac{\partial g_{21}(\xi)}{\partial \xi_{2}} + \frac{1}{2}g^{22}(\xi)\frac{\partial g_{22}(\xi)}{\partial \xi_{2}}\\
\Gamma_{22}^{2}(\xi) &=& 0.
\end{eqnarray*}

\begin{eqnarray*}
\Gamma_{12}^{2}(\xi)&=& \frac{1}{2}g^{21}(\xi)\left(\frac{\partial
g_{21}(\xi)}{\partial \xi_{1}} + \frac{\partial
g_{21}(\xi)}{\partial \xi_{2}} - \frac{\partial
g_{12}(\xi)}{\partial \xi_{2}}\right)+
\frac{1}{2}g^{22}(\xi)\left(\frac{\partial g_{22}(\xi)}{\partial
\xi_{1}}
+ \frac{\partial g_{21}(\xi)}{\partial \xi_{2}} - \frac{\partial g_{12}(\xi)}{\partial \xi_{2}}\right)\\
\Gamma_{12}^{2}(\xi) &=& -\frac{\mu}{\sigma^{2}}.
\end{eqnarray*}

\begin{eqnarray*}
\Gamma_{21}^{2}(\xi)&=& \frac{1}{2}g^{21}(\xi)\left(\frac{\partial
g_{11}(\xi)}{\partial \xi_{1}} + \frac{\partial
g_{22}(\xi)}{\partial \xi_{1}}
 - \frac{\partial g_{21}(\xi)}{\partial \xi_{2}}\right)+ \frac{1}{2}g^{22}(\xi)\left(\frac{\partial g_{12}(\xi)}{\partial \xi_{2}}
 + \frac{\partial g_{22}(\xi)}{\partial \xi_{1}} - \frac{\partial g_{21}(\xi)}{\partial \xi_{2}}\right)\\
\Gamma_{21}^{2}(\xi) &=& \frac{\mu}{\sigma^{2}}.
\end{eqnarray*}

\begin{eqnarray*}
\Gamma_{11}^{2}(\xi)&=& \frac{1}{2}g^{21}(\xi)\left(\frac{\partial
g_{11}(\xi)}{\partial \xi_{1}} + \frac{\partial
g_{21}(\xi)}{\partial \xi_{1}}
 - \frac{\partial g_{11}(\xi)}{\partial \xi_{2}}\right)+ \frac{1}{2}g^{22}(\xi)\left(\frac{\partial g_{12}(\xi)}{\partial \xi_{1}}
 + \frac{\partial g_{21}(\xi)}{\partial \xi_{1}} - \frac{\partial g_{11}(\xi)}{\partial \xi_{2}}\right)\\
\Gamma_{11}^{2}(\xi) &=& \frac{-2\mu^{2} + \mu - 3\sigma^{8} -
6\mu^{2}\sigma^{6}}{\sigma^{8}}.
\end{eqnarray*}

\begin{eqnarray*}
\Gamma_{21}^{1}(\xi)&=& \frac{1}{2}g^{11}(\xi)\left(\frac{\partial
g_{11}(\xi)}{\partial \xi_{2}} + \frac{\partial
g_{12}(\xi)}{\partial \xi_{1}}
 - \frac{\partial g_{21}(\xi)}{\partial \xi_{1}}\right) + \frac{1}{2}g^{12}(\xi)\left(\frac{\partial g_{12}(\xi)}{\partial \xi_{2}}
  + \frac{\partial g_{12}}{\partial \xi_{1}} - \frac{\partial g_{21}(\xi)}{\partial \xi_{1}}\right)\\
\Gamma_{21}^{1}(\xi) &=& \frac{1}{2\sigma^{2}}.
\end{eqnarray*}

\begin{eqnarray*}
\Gamma_{22}^{1}(\xi)&=& \frac{1}{2}g^{11}(\xi)\left(\frac{\partial
g_{21}(\xi)}{\partial \xi_{2}}
 + \frac{\partial g_{12}(\xi)}{\partial \xi_{2}}
 - \frac{\partial g_{22}(\xi)}{\partial \xi_{1}}\right) + \frac{1}{2}g^{12}(\xi)\left(\frac{\partial g_{22}(\xi)}{\partial \xi_{2}}
  + \frac{\partial g_{12}(\xi)}{\partial \xi_{2}} - \frac{\partial g_{22}(\xi)}{\partial \xi_{1}}\right)\\
\Gamma_{22}^{1}(\xi) &=& 0.
\end{eqnarray*}
\begin{eqnarray*}
\Gamma_{12}^{1}(\xi)&=& \frac{1}{2}g^{11}(\xi)\left(\frac{\partial
g_{21}(\xi)}{\partial \xi_{1}} + \frac{\partial
g_{11}(\xi)}{\partial \xi_{2}} -
 \frac{\partial g_{12}(\xi)}{\partial \xi_{1}}\right) + \frac{1}{2}g^{12}(\xi)\left(\frac{\partial g_{22}(\xi)}{\partial \xi_{1}} + \frac{\partial g_{11}(\xi)}{\partial \xi_{2}} - \frac{\partial g_{1'2'}}{\partial \xi^{1}}\right)\\
\Gamma_{12}^{1}(\xi) &=& 8\mu^{2}\sigma^{4} + \frac{4\mu +
1}{2\sigma^{2}}.
\end{eqnarray*}

According to Amari  \cite{Shun}, the tensor of Riemanienne curvature
of Christoffel is given by
\begin{eqnarray}R_{\alpha \beta \gamma \eta} & =& <
R(\partial_{\xi_{\alpha}}, \partial_{\xi_{\beta}},
\partial_{\xi_{\gamma}}),
\partial_{\xi_{\eta}}> \label{a}\end{eqnarray}

we have
\begin{eqnarray*}
R(\partial_{\xi_{\alpha}}, \partial_{\xi_{\beta}},
\partial_{\xi_{\gamma}})
&=&\left(\nabla_{\partial_{\xi_{\alpha}}}\nabla_{\partial_{\xi_{\beta}}}-\nabla_{\partial_{\xi_{\beta}}}\nabla_{\partial_{\xi_{\alpha}}}\right)
\partial_{\xi_{\alpha}} - \nabla_{(\partial_{\xi_{\alpha}}\partial_{\xi_{\beta}}- \partial_{\xi_{\beta}}\partial_{\xi_{\alpha}})}\partial_{\xi_{\gamma}}\\
&=&
\left(\nabla_{\partial_{\xi_{\alpha}}}\nabla_{\partial_{\xi_{\beta}}}(\partial_{\xi_{\gamma}})
-
\nabla_{\partial_{\xi_{\beta}}}\nabla_{\partial_{\xi_{\alpha}}}(\partial_{\xi_{\gamma}})\right)-
\nabla_{\partial_{\xi_{\alpha}}\partial_{\xi_{\beta}}}(\partial_{\xi_{\gamma}})+
\nabla_{\partial_{\xi_{\beta}}\partial_{\xi_{\alpha}}}(\partial_{\xi_{\gamma}})\\
&=&
\left(\nabla_{\partial_{\xi_{\alpha}}}\Gamma_{\beta\gamma}^{s'}\partial_{\xi_{s'}}
- \nabla_{\partial_{\xi_{\beta}}}\Gamma_{\alpha
\gamma}^{s'}\partial_{\xi_{s'}}\right) -
\partial_{\alpha}\nabla_{\partial_{\xi_{\beta}}}(\partial_{\xi_{\gamma}})+
\partial_{\xi_{\beta}}\nabla_{\partial_{\xi_{\alpha}}}(\partial_{\xi_{\gamma}})\\
&=&
\left(\partial_{\xi_{\alpha}}\Gamma_{\beta\gamma}^{s'}\partial_{\xi_{s'}}
+ \Gamma_{\beta\gamma
}^{s'}\nabla_{\partial_{\xi_{\alpha}}}\partial_{\xi_{s'}}\right)-
\left(
\partial_{\xi_{\beta}}\Gamma_{\alpha\gamma}^{s'}\partial_{\xi_{s'}} +
\Gamma_{\alpha\gamma
}^{s'}\nabla_{\partial_{\xi_{\beta}}}\partial_{\xi_{s'}}\right)\\
&& -
\partial_{\xi_{\alpha}}\nabla_{\partial_{\xi_{\beta}}}(\partial_{\xi_{\gamma}}) +
\partial_{\xi_{\beta}}\nabla_{\partial_{\xi_{\alpha}}}(\partial_{\xi_{\gamma}})\\
&=&
\left(\partial_{\xi_{\alpha}}\Gamma_{\beta\gamma}^{s'}\partial_{\xi_{s'}}
+ \Gamma_{\beta\gamma
}^{s'}\nabla_{\partial_{\xi_{\alpha}}}\partial_{\xi_{s'}}\right)-
\left(
\partial_{\xi_{\beta}}\Gamma_{\alpha\gamma}^{s'}\partial_{\xi_{s'}} +
\Gamma_{\alpha\gamma
}^{s'}\nabla_{\partial_{\xi_{\beta}}}\partial_{\xi_{s'}}\right)\\
&& -
\partial_{\xi_{\alpha}}\Gamma_{\beta \gamma}^{r'}\partial_{\xi_{r'}} +
\partial_{\xi_{\beta}}\Gamma_{\alpha\gamma}^{r'}\partial_{\xi_{r'}}
\end{eqnarray*}

While replacing in (\ref{a}) we obtain

\begin{eqnarray*}R_{\alpha \beta \gamma \eta} &=& \left(
\partial_{\xi_{\alpha}}
\Gamma^{s'}_{\beta\gamma}-
\partial_{\xi_{\beta}}\Gamma^{s'}_{\alpha \gamma}\right)g_{s'\eta} -
\left( \Gamma^{r'}_{\beta\gamma} \Gamma_{\alpha r' \eta} -
\Gamma^{r'}_{\alpha \gamma}\Gamma_{\beta r' \eta}\right)\nonumber\\
&&+ <\Gamma^{s'}_{\beta
\gamma}\nabla_{\partial_{\xi_{\alpha}}}\partial_{\xi_{s'}},
\partial_{\xi_{\eta}}> - <\Gamma^{s'}_{\alpha
\gamma}\nabla_{\partial_{\xi_{\beta}}}\partial_{\xi_{s'}},
\partial_{\xi_{\eta}}>
\end{eqnarray*}
with $\alpha, \beta, \gamma, \eta, s'\in\left\{1,2\right\}$ so,

\begin{eqnarray*}
R_{1111}(\xi)&=& \left(\partial_{\xi_{1}}\Gamma_{11}^{1}(\xi) - \partial_{\xi_{1}}\Gamma_{11}^{1}(\xi)\right)g_{11}(\xi)+ \left(\partial_{\xi_{1}}\Gamma_{11}^{2}(\xi) - \partial_{\xi_{1}}\Gamma_{11}^{2}(\xi)\right)g_{21}(\xi)\\
&& - \left(\Gamma_{111}(\xi)\Gamma_{11}^{1}(\xi) - \Gamma_{111}(\xi)\Gamma_{11}^{1}(\xi)\right)- \left(\Gamma_{121}(\xi)\Gamma_{11}^{2}(\xi) - \Gamma_{121}(\xi)\Gamma_{11}^{2}(\xi)\right) \\
&&+ <\Gamma_{11}^{1}(\xi)\nabla_{\partial_{\xi_{1}}}(\partial_{\xi_{1}}), \partial_{\xi_{1}}>  + <\Gamma_{11}^{2}(\xi)\nabla_{\partial_{\xi_{1}}}(\partial_{\xi_{2}}), \partial_{\xi_{1}}>\\
&& -<\Gamma_{11}^{1}(\xi)\nabla_{\partial_{\xi_{1}}}(\partial_{\xi_{1}}), \partial_{\xi_{1}}> -<\Gamma_{11}^{2}\nabla_{\partial_{\xi_{2}}}(\partial_{\xi_{1}}), \partial_{\xi_{1}}> \\
&=& 0
\end{eqnarray*}

\begin{eqnarray*}
R_{2222}(\xi)&=& \left(\partial_{\xi_{2}}\Gamma_{22}^{1}(\xi) - \partial_{\xi_{2}}\Gamma_{22}^{1}(\xi)\right)g_{12}(\xi)+ \left(\partial_{\xi_{2}}\Gamma_{22}^{2}(\xi) - \partial_{\xi_{2}}\Gamma_{22}^{2}(\xi)\right)g_{22}(\xi)  \\
&&- \left(\Gamma_{222}(\xi)\Gamma_{22}^{1}(\xi) - \Gamma_{222}(\xi)\Gamma_{22}^{1}(\xi)\right)- \left(\Gamma_{222}(\xi)\Gamma_{22}^{2}(\xi) - \Gamma_{222}(\xi)\Gamma_{22}^{2}(\xi)\right)\\
&&+ <\Gamma_{22}^{1}(\xi)\nabla_{\partial_{\xi_{2}}}(\partial_{\xi_{1}}), \partial_{\xi_{2}}> + <\Gamma_{22}^{2}(\xi)\nabla_{\partial_{\xi_{2}}}(\partial_{\xi_{2}}), \partial_{\xi_{2}}> \\
&& -<\Gamma_{22}^{1}(\xi)\nabla_{\partial_{\xi_{2}}}(\partial_{\xi_{1}}),\partial_{\xi_{2}}> -<\Gamma_{22}^{2}(\xi)\nabla_{\partial_{\xi_{2}}}(\partial_{\xi_{2}}),\partial_{\xi_{2}}>\\
 &=& 0
\end{eqnarray*}

\begin{eqnarray*}
R_{1122}(\xi)&=& \left(\partial_{\xi_{1}}\Gamma_{12}^{1}(\xi) - \partial_{\xi_{1}}\Gamma_{12}^{1}(\xi)\right)g_{12}(\xi)+ \left(\partial_{\xi_{1}}\Gamma_{12}^{2}(\xi) - \partial_{1'}\Gamma_{12}^{2}(\xi)\right)g_{22}(\xi) \\
&&- \left(\Gamma_{112}(\xi)\Gamma_{12}^{1}(\xi) - \Gamma_{112}(\theta)\Gamma_{12}^{1}(\xi)\right)- \left(\Gamma_{122}(\xi)\Gamma_{12}^{2}(\xi) - \Gamma_{122}(\xi)\Gamma_{12}^{2}(\xi)\right) \\
&&+ <\Gamma_{12}^{1}(\xi)\nabla_{\partial_{\xi_{1}}}(\partial_{\xi_{1}}), \partial_{\xi_{2}}> + <\Gamma_{12}^{2}(\xi)\nabla_{\partial_{\xi_{1}}}(\partial_{\xi_{2}}), \partial_{\xi_{2}}>  \\
&&-<\Gamma_{12}^{1}(\xi)\nabla_{\partial_{\xi_{2}}}(\partial_{\xi_{1}}), \partial_{\xi_{2}}> - <\Gamma_{12}^{2}(\xi)\nabla_{\partial_{\xi_{2}}}(\partial_{\xi_{2}}),\partial_{\xi_{2}}>\\
 &=&  0.
\end{eqnarray*}

\begin{eqnarray*}
R_{1112}(\xi)&=& \left(\partial_{\xi_{1}}\Gamma_{11}^{1}(\xi) - \partial_{\xi_{1}}\Gamma_{11}^{1}(\xi)\right)g_{12}(\xi)+ \left(\partial_{\xi_{1}}\Gamma_{11}^{2}(\xi) - \partial_{1'}\Gamma_{11}^{2}(\xi)\right)g_{22}(\xi)\\
&& - \left(\Gamma_{112}(\xi)\Gamma_{11}^{1}(\xi) - \Gamma_{112}(\xi)\Gamma_{11}^{1}(\xi)\right)- \left(\Gamma_{122}(\xi)\Gamma_{11}^{2}(\xi) - \Gamma_{122}(\xi)\Gamma_{11}^{2}(\xi)\right) \\
&&+ <\Gamma_{11}^{1}(\xi)\nabla_{\partial_{\xi_{1}}}(\partial_{\xi_{1}}), \partial_{\xi_{2}}> + <\Gamma_{11}^{2}(\xi)\nabla_{\partial_{\xi_{1}}}(\partial_{\xi_{2}}), \partial_{\xi_{2}}> \\
&&-<\Gamma_{11}^{1}(\xi)\nabla_{\partial_{\xi_{1}}}(\partial_{\xi_{1}}), \partial_{\xi_{2}}> -<\Gamma_{11}^{2}(\xi)\nabla_{\partial_{\xi_{1}}}(\partial_{\xi_{2}}), \partial_{\xi_{2}}> \\
 &=&0
\end{eqnarray*}

\begin{eqnarray*}
R_{2121}(\xi)&=& \left(\partial_{2'}\Gamma_{12}^{1}(\xi) - \partial_{\xi_{1}}\Gamma_{22}^{1}(\xi)\right)g_{11}(\xi)+ \left(\partial_{\xi_{2}}\Gamma_{12}^{2}(\xi) - \partial_{\xi_{1}}\Gamma_{22}^{2}(\xi)\right)g_{21}(\xi)  \\
&&- \left(\Gamma_{211}(\xi)\Gamma_{12}^{1}(\xi) - \Gamma_{111}(\xi)\Gamma_{22}^{1}(\xi)\right)- \left(\Gamma_{221}(\xi)\Gamma_{12}^{2}(\xi) - \Gamma_{121}(\xi)\Gamma_{22}^{2}(\xi)\right)\\
&&+ <\Gamma_{12}^{1}(\xi)\nabla_{\partial_{\xi_{2}}}(\partial_{\xi_{1}}), \partial_{\xi_{1}}> + <\Gamma_{12}^{2}(\xi)\nabla_{\partial_{\xi_{2}}}(\partial_{\xi_{2}}), \partial_{\xi_{1}}> \\
&&- <\Gamma_{22}^{1}(\xi)\nabla_{\partial_{\xi_{1}}}(\partial_{\xi_{1}}), \partial_{\xi_{1}}> -<\Gamma_{22}^{2}(\xi)\nabla_{\partial_{\xi_{1}}}(\partial_{\xi_{2}}), \partial_{\xi_{1}}> \\
&=& 0
\end{eqnarray*}

\begin{eqnarray*}
R_{2211}(\xi) &=& \left(\partial_{\xi_{2}}\Gamma_{21}^{1}(\xi) - \partial_{\xi_{2}}\Gamma_{21}^{1}(\xi)\right)g_{11}(\xi)+ \left(\partial_{\xi_{2}}\Gamma_{21}^{2}(\xi) - \partial_{\xi_{2}}\Gamma_{21}^{2}(\xi)\right)g_{21}(\xi) \\
&&- \left(\Gamma_{211}(\xi)\Gamma_{21}^{1}(\xi) - \Gamma_{211}\Gamma_{21}^{1}(\xi)\right)- \left(\Gamma_{221}(\xi)\Gamma_{21}^{2}(\xi) - \Gamma_{221}(\xi)\Gamma_{21}^{2}(\xi)\right)\\
&& + \Gamma_{211}(\xi)\Gamma_{21}^{1}(\xi) + \Gamma_{221}(\xi)\Gamma_{21}^{2}(\xi) - \Gamma_{211}(\xi)\Gamma_{21}^{1}(\xi) - \Gamma_{221}(\xi)\Gamma_{21}^{2}(\xi)\\
&=& 0.
\end{eqnarray*}

\begin{eqnarray*}
R_{2221}(\xi) &=& \left(\partial_{\xi_{2}}\Gamma_{22}^{1}(\xi) - \partial_{\xi_{2}}\Gamma_{22}^{1}(\xi)\right)g_{11}(\xi)+ \left(\partial_{\xi_{2}}\Gamma_{22}^{2}(\xi) - \partial_{\xi_{2}}\Gamma_{22}^{2}(\xi)\right)g_{21}(\xi) \\
&&- \left(\Gamma_{211}(\xi)\Gamma_{22}^{1}(\xi) - \Gamma_{211}(\xi)\Gamma_{22}^{1}(\xi)\right)- \left(\Gamma_{211}(\xi)\Gamma_{22}^{1}(\xi) - \Gamma_{211}(\xi)\Gamma_{22}^{1}(\xi)\right)\\
&&+  \Gamma_{211}(\xi)\Gamma_{22}^{1}(\xi) + \Gamma_{221}\Gamma_{22}^{2}(\xi) - \Gamma_{211}(\xi)\Gamma_{22}^{1}(\xi) - \Gamma_{221}(\xi)\Gamma_{22}^{2}(\xi)\\
&=& 0.
\end{eqnarray*}

\begin{eqnarray*}
R_{2212}(\xi) &=& \left(\partial_{\xi_{2}}\Gamma_{21}^{1}(\xi) - \partial_{\xi_{2}}\Gamma_{21}^{1}(\xi)\right)g_{12}(\xi)+ \left(\partial_{\xi_{2}}\Gamma_{21}^{2}(\xi) - \partial_{\xi_{2}}\Gamma_{21}^{2}(\xi)\right)g_{22}(\xi) \\
&&- \left(\Gamma_{212}(\xi)\Gamma_{21}^{1}(\xi) - \Gamma_{212}(\xi)\Gamma_{21}^{1}(\xi)\right)- \left(\Gamma_{222}\Gamma_{21}^{2}(\xi) - \Gamma_{222}(\xi)\Gamma_{21}^{2}(\xi)\right)\\
&&+ \Gamma_{211}\Gamma_{21}^{1}(\xi) + \Gamma_{221}(\xi)\Gamma_{21}^{2}(\xi) - \Gamma_{211}\Gamma_{21}^{1}(\xi) - \Gamma_{221}(\xi)\Gamma_{21}^{2}(\xi)\\
&=& 0.
\end{eqnarray*}

\begin{eqnarray*}
R_{2122}(\xi) &=& \left(\partial_{\xi_{2}}\Gamma_{12}^{1}(\xi) - \partial_{\xi_{1}}\Gamma_{22}^{1}(\xi)\right)g_{12}(\xi)+ \left(\partial_{\xi_{2}}\Gamma_{12}^{2}(\xi) - \partial_{\xi_{1}}\Gamma_{22}^{2}(\xi)\right)g_{22}(\xi)\\
&& - \left(\Gamma_{212}(\xi)\Gamma_{12}^{1}(\xi) - \Gamma_{112}(\xi)\Gamma_{22}^{1}(\xi)\right)- \left(\Gamma_{222}(\xi)\Gamma_{21}^{2}(\xi) - \Gamma_{122}(\xi)\Gamma_{22}^{2}(\xi)\right)\\
&& + \Gamma_{212}(\xi)\Gamma_{12}^{1}(\xi) + \Gamma_{222}(\xi)\Gamma_{21}^{2}(\xi) - \Gamma_{211}\Gamma_{22}^{1}(\xi) - \Gamma_{122}(\xi)\Gamma_{22}^{2}(\xi)\\
&=& 0.
\end{eqnarray*}

\begin{eqnarray*}
R_{1221}(\xi) &=& \left(\partial_{\xi_{1}}\Gamma_{22}^{1}(\xi) - \partial_{\xi_{2}}\Gamma_{12}^{1}(\xi)\right)g_{11}(\xi)+ \left(\partial_{\xi_{1}}\Gamma_{22}^{2}(\xi) - \partial_{\xi_{2}}\Gamma_{12}^{2}(\xi)\right)g_{21}(\xi) \\
&&-  \left(\Gamma_{111}\Gamma_{22}^{1}(\xi) - \Gamma_{211}(\xi)\Gamma_{21}^{1}(\xi)\right)- \left(\Gamma_{121}(\xi)\Gamma_{22}^{2}(\xi) - \Gamma_{221}(\xi)\Gamma_{12}^{2}(\xi)\right)\\
&& + \Gamma_{111}(\xi)\Gamma_{21}^{1}(\xi) + \Gamma_{121}(\xi)\Gamma_{21}^{2}(\xi) - \Gamma_{211}(\xi)\Gamma_{12}^{1}(\xi) - \Gamma_{221}\Gamma_{12}^{2}(\xi)\\
&=& \Gamma_{111}(\xi)\Gamma_{21}^{1}(\xi) + \Gamma_{121}(\xi)\Gamma_{21}^{2}(\xi) \\
&=& \frac{\mu}{\sigma^{6}}.
\end{eqnarray*}

\begin{eqnarray*}
R_{2112}(\xi) &=& \left(\partial_{\xi_{2}}\Gamma_{11}^{1}(\xi) - \partial_{\xi_{1}}\Gamma_{21}^{1}(\xi)\right)g_{12}(\xi)+ \left(\partial_{\xi_{2}}\Gamma_{11}^{2}(\xi) - \partial_{\xi_{1}}\Gamma_{21}^{2}(\xi)\right)g_{22}(\xi) \\
&&-  \left(\Gamma_{212}(\xi)\Gamma_{21}^{1}(\xi) - \Gamma_{112}(\xi)\Gamma_{21}^{1}(\xi)\right)- \left(\Gamma_{222}(\xi)\Gamma_{21}^{2}(\xi) - \Gamma_{222}(\xi)\Gamma_{21}^{2}(\xi)\right)\\
&& + \Gamma_{212}(\xi)\Gamma_{11}^{1}(\xi) + \Gamma_{222}(\xi)\Gamma_{11}^{2}(\xi) - \Gamma_{112}(\xi)\Gamma_{21}^{1}(\xi) - \Gamma_{122}(\xi)\Gamma_{21}^{2}(\xi)\\
&=& \frac{6\sigma^{8} + 3\mu\sigma^{6} + 6 \mu^{2}\sigma^{6}+ 6
\mu^{2}}{4 \sigma^{14}}.
\end{eqnarray*}

\begin{eqnarray*}
R_{1212}(\xi)&=& \left(\partial_{\xi_{1}}\Gamma_{21}^{1}(\xi) - \partial_{\xi_{2}}\Gamma_{11}^{1}(\xi)\right)g_{12}(\xi)+ \left(\partial_{\xi_{1}}\Gamma_{12}^{2}(\xi) - \partial_{\xi_{2}}\Gamma_{11}^{2}(\xi)\right)g_{22}(\xi) \\
&&- \left(\Gamma_{112}(\xi)\Gamma_{21}^{1}(\xi) - \Gamma_{212}(\xi)\Gamma_{11}^{1}(\xi)\right)- \left(\Gamma_{122}(\xi)\Gamma_{21}^{2}(\xi) - \Gamma_{222}(\xi)\Gamma_{11}^{2}(\xi)\right) \\
&&+ <\Gamma_{21}^{1}(\xi)\nabla_{\partial_{\xi_{1}}}(\partial_{\xi_{1}}), \partial_{\xi_{2}}> + <\Gamma_{21}^{2}(\xi)\nabla_{\partial_{\xi_{1}}}(\partial_{\xi_{2}}), \partial_{\xi_{2}}>  \\
&&- <\Gamma_{11}^{1}(\xi)\nabla_{\partial_{\xi_{1}}}(\partial_{\xi_{1}}), \partial_{\xi_{2}}> -<\Gamma_{11}^{2}\nabla_{\partial_{\xi_{2}}}(\partial_{\xi_{2}}), \partial_{\xi_{2}}>\\
&&= \left(\partial_{\xi_{1}}\Gamma_{21}^{1}(\xi) - \partial_{\xi_{2}}\Gamma_{11}^{1}\right)g_{12}(\xi)+ \left(\partial_{\xi_{1}}\Gamma_{12}^{2}(\xi) - \partial_{\xi_{2}}\Gamma_{11}^{2}(\xi)\right)g_{22}(\xi) \\
&&- \left(\Gamma_{112}(\xi)\Gamma_{21}^{1}(\xi) - \Gamma_{212}\Gamma_{11}^{1}\right)- \left(\Gamma_{122}\Gamma_{21}^{2} - \Gamma_{222}(\xi)\Gamma_{11}^{2}(\xi)\right)\\
&& + \Gamma_{112}(\xi)\Gamma_{21}^{1}(\xi) + \Gamma_{122}(\xi)\Gamma_{21}^{2}(\xi) - \Gamma_{212}(\xi)\Gamma_{11}^{1}(\xi) - \Gamma_{222}(\xi)\Gamma_{11}^{2}(\xi)\\
&=& \frac{1}{2\sigma^{6}}.
\end{eqnarray*}

\begin{eqnarray*}
R_{1211}(\xi)&=& \left(\partial_{\xi_{1}}\Gamma_{21}^{1}(\xi) - \partial_{\xi_{2}}\Gamma_{11}^{1}(\xi)\right)g_{11}(\xi)+ \left(\partial_{\xi_{2}}\Gamma_{21}^{2}(\xi) - \partial_{\xi_{2}}\Gamma_{11}^{2}(\xi)\right)g_{21}(\xi)  \\
&&- \left(\Gamma_{111}(\xi)\Gamma_{21}^{1}(\xi) - \Gamma_{211}(\xi)\Gamma_{11}^{1}(\xi)\right)- \left(\Gamma_{121}\Gamma_{21}^{2}(\xi) - \Gamma_{221}(\xi)\Gamma_{11}^{2}(\xi)\right)\\
&&+ \Gamma_{111}(\xi)\Gamma_{21}^{1}(\xi) + \Gamma_{121}(\xi)\Gamma_{21}^{2}(\xi) - \Gamma_{221}(\xi)\Gamma_{11}^{2}(\xi) - \Gamma_{211}(\xi)\Gamma_{11}^{1}(\xi)\\
&=& 0.
\end{eqnarray*}

\begin{eqnarray*}
R_{1121}(\xi)&=& \left(\partial_{\xi_{1}}\Gamma_{12}^{1}(\xi) - \partial_{\xi_{1}}\Gamma_{12}^{1}(\xi)\right)g_{11}(\xi)+ \left(\partial_{\xi_{1}}\Gamma_{12}^{2}(\xi) - \partial_{\xi_{1}}\Gamma_{12}^{2}(\xi)\right)g_{21}(\xi)\\
&& - \left(\Gamma_{111}(\xi)\Gamma_{12}^{1}(\xi) - \Gamma_{111}(\xi)\Gamma_{12}^{1}(\xi)\right)- \left(\Gamma_{121}\Gamma_{12}^{2}(\xi) - \Gamma_{121}(\xi)\Gamma_{12}^{2}(\xi)\right)\\
&& + \Gamma_{111}(\xi)\Gamma_{12}^{1}(\xi) + \Gamma_{121}(\xi)\Gamma_{12}^{2}(\xi) - \Gamma_{111}(\xi)\Gamma_{12}^{1}(\xi) - \Gamma_{121}(\xi)\Gamma_{12}^{2}(\xi)\\
&=& 0.
\end{eqnarray*}

\begin{eqnarray*}
R_{2111}(\xi)&=& \left(\partial_{\xi_{2}}\Gamma_{11}^{1}(\xi) - \partial_{\xi_{1}}\Gamma_{11}^{1}(\xi)\right)g_{11}(\xi)+ \left(\partial_{\xi_{2}}\Gamma_{11}^{2}(\xi) - \partial_{\xi_{1}}\Gamma_{11}^{2}(\xi)\right)g_{21}(\xi) \\
&&- \left(\Gamma_{211}(\xi)\Gamma_{11}^{1}(\xi) - \Gamma_{112}(\xi)\Gamma_{21}^{1}(\xi)\right)- \left(\Gamma_{221}(\xi)\Gamma_{11}^{2}(\xi) - \Gamma_{122}(\xi)\Gamma_{21}^{2}(\xi)\right) \\
&&+ <\Gamma_{11}^{1}(\xi)\nabla_{\partial_{\xi_{1}}}(\partial_{\xi_{1}}), \partial_{\xi_{1}}> + <\Gamma_{11}^{2}(\xi)\nabla_{\partial_{\xi_{1}}}(\partial_{\xi_{2}}), \partial_{\xi_{1}}> \\
&&-<\Gamma_{21}^{1}(\xi)\nabla_{\partial_{\xi_{1}}}(\partial_{\xi_{1}}), \partial_{\xi_{2}}> -<\Gamma_{21}^{2}(\xi)\nabla_{\partial_{\xi_{1}}}(\partial_{\xi_{2}}), \partial_{\xi_{2}}> \\
&=& \left(\partial_{2'}\Gamma_{11}^{1}(\xi) - \partial_{\xi_{1}}\Gamma_{11}^{1}(\xi)\right)g_{11}(\xi)+ \left(\partial_{\xi_{2}}\Gamma_{11}^{2}(\xi) - \partial_{\xi_{1}}\Gamma_{11}^{2}(\xi)\right)g_{21}(\xi) \\
&&- \left(\Gamma_{211}\Gamma_{11}^{1}(\xi) - \Gamma_{112}\Gamma_{21}^{1}(\xi)\right)- \left(\Gamma_{221}(\xi)\Gamma_{11}^{2}(\xi) - \Gamma_{122}(\xi)\Gamma_{21}^{2}(\xi)\right)\\
&& + \Gamma_{11}^{1}(\xi)\Gamma_{111}(\xi) + \Gamma_{11}^{1}(\xi)\Gamma_{121}(\xi) -\Gamma_{21}^{1}(\xi) - \Gamma_{21}^{2}(\xi)\Gamma_{122}(\xi)\\
&=& \frac{-2\sigma^{8} + 2 \mu^{2}\sigma^{6} + 46 \mu^{2}\sigma^{8}
+ 24 \mu^{4}\sigma^{6} + 12 \mu^{2}\sigma^{2} - 4\mu \sigma^{2} +
6\sigma^{10} + 12\mu^{4} - 3\mu^{2} + 9 \mu \sigma^{8} + 18
\mu^{3}\sigma^{6}}{2\sigma^{14}}
\end{eqnarray*}

Using the definition \ref{def1}, we have
 \begin{eqnarray*} \mathcal{K}&=&
\frac{1}{2}R_{1111}(\xi)g^{11}(\xi)g^{11}(\xi) +
\frac{1}{2}R_{1112}(\xi)g^{12}(\xi)g^{11}(\xi) +
\frac{1}{2}R_{1121}(\xi)g^{11}(\xi)g^{12}(\xi)\\
&&+\frac{1}{2}R_{1122}(\xi)g^{12}(\xi)g^{12}(\xi) + \frac{1}{2}R_{1211}(\xi)g^{11}(\xi)g^{21}(\xi) + \frac{1}{2}R_{1212}(\xi)g^{12}(\xi)g^{21}(\xi)\\
&& + \frac{1}{2}R_{1221}(\xi)g^{11}(\xi)g^{22}(\xi) + \frac{1}{2}R_{1222}(\xi)g^{12}(\xi)g^{22}(\xi) + \frac{1}{2}R_{2111}(\xi)g^{21}(\xi)g^{11}(\xi) \\
&&+ \frac{1}{2}R_{2112}(\xi)g^{22}(\xi)g^{11}(\xi) + \frac{1}{2}R_{2121}(\xi)g^{21}(\xi)g^{12}(\xi) + \frac{1}{2}R_{2122}(\xi)g^{22}(\xi)g^{12}(\xi) \\
&&+ \frac{1}{2}R_{2211}(\xi)g^{21}(\xi)g^{21}(\xi) + \frac{1}{2}R_{2212}(\xi)g^{22}g^{21}(\xi) + \frac{1}{2}R_{2221}(\xi)g^{21}(\xi)g^{22}(\xi)\\
&& + \frac{1}{2}R_{2222}(\xi)g^{22}(\xi)g^{22}(\xi)\\
&=&   \frac{1}{2}R_{1212}(\xi)g^{12}(\xi)g^{21}(\xi) + \frac{1}{2}R_{2111}(\xi)g^{21}(\xi)g^{11}(\xi) + \frac{1}{2}R_{2112}(\xi)g^{22}(\xi)g^{11}(\xi)\\
&=&  \frac{16\mu^{2}\sigma^{12} + 12\mu^{4}\sigma^{4} + 6\sigma^{14} - \mu\sigma^{12} - 2 \mu^{2}\sigma^{6} +10\mu^{3}\sigma^{10} +92 \mu^{3}\sigma^{12} +48\mu^{5}\sigma^{10}}{4\sigma^{14}}\\
&&
+\frac{12\mu\sigma^{14}+24\mu^{5}\sigma^{4}-6\mu^{3}\sigma^{4}}{4\sigma^{14}}
\end{eqnarray*}

Thus the curve will be given by

\begin{eqnarray*}\mathcal{K}&=& \frac{16\mu^{2}\sigma^{8} + 12\mu^{4} + 6\sigma^{10} - \mu\sigma^{8} - 2 \mu^{2}\sigma^{2} +10\mu^{3}\sigma^{6} +92 \mu^{3}\sigma^{8} +48\mu^{5}\sigma^{6}}{4\sigma^{10}}\\
&& + \frac{12\mu\sigma^{10}+24\mu^{5}-6\mu^{3}}{4\sigma^{10}}
\end{eqnarray*}

\end{proof}

\section{General conclusion}\label{sec6}

This study has examined the differential geometric structure of the
Gaussian statistical model when expressed in a dual coordinate
system $\xi = (\mu, \mu^2 + \sigma^2)$. Through explicit
computations of the Fisher information metric, the affine
connection, and the curvature tensors, we have shown that the dual
parametrization introduces non-trivial geometric effects. In
contrast to the classical coordinate system $\theta = (\mu,
\sigma)$, where the manifold is known to be both flat and
torsion-free, the $\xi$-system exhibits non zero torsion and scalar
curvature. These findings highlight an important phenomenon in
information geometry: duality does not guarantee invariance of
geometric quantities such as curvature and torsion. In particular,
our results show that the choice of coordinate system can deeply
influence the global geometric interpretation of a statistical
model, even for well understood cases like the normal distribution.
From a broader perspective, this work invites further exploration
into the stability and variability of geometric invariants under
coordinate transformations. Understanding which geometric properties
are intrinsic to the model and which depend on parametrization
remains a central question. Future research could aim to classify
transformations that preserve curvature or to extend the analysis to
other statistical families, such as the beta, Weibull, or log-normal
distributions. Ultimately, this study reinforces the idea that the
geometry of statistical models is not only local and computational
it is also deeply sensitive to the way we represent information.

\backmatter

\bmhead{Supplementary information} This manuscript has no additional
data.

\bmhead{Acknowledgments} We would like to thank all the active
members of the algebra and geometry research group at the University
of Maroua in Cameroon.

\section*{Declarations}
This article has no conflict of interest to the journal. No
financing with a third party.
\begin{itemize}
\item No Funding
\item No Conflict of interest/Competing interests (check journal-specific guidelines for which heading to use)
\item  Ethics approval
\item  Consent to participate
\item  Consent for publication
\item  Availability of data and materials
\item  Code availability
\item Authors' contributions
\end{itemize}
\bibliography{bibliography}

\end{document}